\documentclass[11pt]{article}
\usepackage[numbers,square]{natbib}
\usepackage{amsmath}
\usepackage{amsthm}
\usepackage{amsfonts}
\usepackage{amssymb}
\usepackage{siunitx}
\usepackage{commath}
\usepackage{graphicx}
\usepackage{xspace}
\usepackage{color}
\usepackage[lofdepth,lotdepth]{subfig}
\usepackage{stmaryrd}
\usepackage{url}
\usepackage{amsthm}
\usepackage{graphbox}
\usepackage{footnote}
\usepackage{multirow}
\makesavenoteenv{tabular}
\makesavenoteenv{table}
\usepackage[hidelinks]{hyperref}
\hypersetup{colorlinks = true, urlcolor = blue,
  linkcolor = blue, citecolor = blue}
\usepackage{cleveref}

\usepackage[margin=0.7in]{geometry}
\makeatletter
\newcommand{\tnorm}{\@ifstar\@tnorms\@tnorm}
\newcommand{\@tnorms}[1]{%
  \left|\mkern-1.5mu\left|\mkern-1.5mu\left|
   #1
  \right|\mkern-1.5mu\right|\mkern-1.5mu\right|
}
\newcommand{\@tnorm}[2][]{%
  \mathopen{#1|\mkern-1.5mu#1|\mkern-1.5mu#1|}
  #2
  \mathclose{#1|\mkern-1.5mu#1|\mkern-1.5mu#1|}
}
\makeatother

\newtheorem{theorem}{Theorem}
\newtheorem{lemma}{Lemma}

\title{Preconditioning of a hybridizable discontinuous Galerkin
  method for Biot's consolidation model}
\author{
  E. Henr\'iquez\thanks{Department of Applied Mathematics, University of
    Waterloo, ON, Canada (\url{ehenriqu@uwaterloo.ca}),
    \url{http://orcid.org/0000-0002-0243-0368}}
  \and
  J. J. Lee\thanks{Department of Mathematics, Baylor University,
    TX, USA (\url{jeonghun_lee@baylor.edu}),
    \url{https://orcid.org/0000-0001-5201-8526}}
  \and
  S. Rhebergen\thanks{Department of Applied Mathematics, University of
    Waterloo, ON, Canada (\url{srheberg@uwaterloo.ca}),
    \url{http://orcid.org/0000-0001-6036-0356}}}
\begin{document}
\maketitle
\begin{abstract}
  We present a parameter-robust preconditioner for a hybridizable
  discontinuous Galerkin (HDG) discretization of a four-field
  formulation of Biot's consolidation model. We first determine a
  parameter-robust preconditioner for the full discretization. HDG
  methods, however, allow for static condensation. We therefore apply
  the framework presented in our previous work [arXiv:2503.05918,
  2025] to show that a reduced form of the preconditioner is also
  parameter-robust for the reduced HDG discretization. We verify the
  parameter-robustness of the preconditioner through numerical
  examples in both two and three dimensions.  
\end{abstract}
\section{Introduction}
\label{s:introduction}
Biot's consolidation model has applications in, for example, medical
research \cite{fellah2008application,guo2018subject}, geophysics
\cite{coussy2004poromechanics,teatini2006groundwater}, and engineering
\cite{pao2001fully,ferronato2004radioactive}.  Therefore, finding
numerical solutions of this problem is an active area of research,
both in the design of efficient discretization methods, see for
example,
\cite{boffi2016nonconforming,feng2018analysis,kanschat2018finite,lee2016robust,oyarzua2016locking,yi2014convergence},
and solvers, for example,
\cite{bean2017block,cai2023some,cai2015comparisons,chu2025block,kadeethum2020finite}. In
this paper we consider preconditioners for a hybridizable
discontinuous Galerkin (HDG) discretization of Biot's model. A
significant challenge in designing preconditioners for Biot's model is
that the many parameters in Biot's model can vary across several
orders of magnitude. The objective is therefore to design a
preconditioner that is robust with respect to such parameter
variations.

A general framework for parameter-robust preconditioners was presented
by Mardal and Winther \cite{mardal2011preconditioning}. This framework
has been used to design parameter-robust preconditioners for various
formulations of the Biot equations. For example,
\cite{lee2017parameter} develops block-diagonal preconditioners based
on the Mardal and Winther framework for a conforming finite element
discretization of the three-field (displacement, total pressure, fluid
pressure) formulation. The same approach was used in
\cite{lee2023analysis} to design parameter-robust preconditioners for
the displacement, total pressure, and fluid pressure formulation of
mixed finite element and stabilized methods for the Biot equations. In
\cite{chen2020robust}, the authors present parameter-robust
preconditioners of both a conforming stabilized two-field formulation
(displacement and fluid pressure) and a conforming three-field
formulation (displacement, fluid flux, and fluid pressure) of Biot's
equations using Schur complement techniques in combination with the
Mardal and Winther framework. Furthermore, \cite{rodrigo2024parameter}
introduces both block-diagonal and block-triangular preconditioners
for Biot's equation for a non-conforming three-field discretization
and a conforming stabilized three-field formulation. Additional works
related to the design of parameter-robust preconditioners include
\cite{baerland2017weakly,boon2021robust,frigo2021efficient,kraus2021uniformly,luber2024robust},
among others.

In this paper, we consider preconditioners for an HDG discretization
of the Biot equations. HDG methods were introduced by Cockburn et al.
\cite{cockburn2009hybridizable} as a computationally efficient variant
of the discontinuous Galerkin (DG) method. Computational efficiency is
achieved by static condensation in which cell degrees-of-freedom are
eliminated from the linear system; the global linear system is one for
trace unknowns only. We, in particular, consider the HDG
discretization for a four-field (displacement, fluid pressure, total
pressure, Darcy velocity) formulation of Biot’s equations that was
presented and analyzed in \cite{cesmelioglu2023analysis}. This
discretization enforces $H(\text{div})$-conformity through Lagrange
multipliers (the pressure trace unknowns), and is such that the
compressibility equation is satisfied pointwise within the elements
and mass conservation is achieved pointwise up to the projection error
of the source term. Furthermore, due to the introduction of the total
pressure and Darcy velocity, the discretization is free of volumetric
locking. Related HDG discretizations, but in which
$H(\text{div})$-conformity is enforced directly by the finite element
spaces, are presented for a two-field (displacement, pressure)
formulation of the Biot equations in \cite{fu2019high} and for a
three-field (displacement, fluid pressure, Darcy velocity) formulation
in \cite{kraus2021uniformly}. The latter work also presents a
parameter-robust preconditioner for the non-reduced linear system. A
reduced form of the preconditioner is subsequently used to solve the
statically condensed problem. An analysis of the robustness of the
reduced preconditioner, however, is not presented. An extension of
\cite{kraus2021uniformly} is presented in \cite{kraus2023hybridized}
for the quasi-static multiple-network poroelastic theory (MPET) model.

In our previous work \cite{henriquez2025parameter} we extended the
parameter-robust framework of Mardal and Winther to hybridizable
discretizations. In particular, we identified a condition that if
satisfied, then the reduced form of a parameter-robust preconditioner
for the full system will be parameter-robust for the statically
condensed problem. This framework will now be used to determine a
parameter-robust preconditioner for the HDG discretization presented
in \cite{cesmelioglu2023analysis} for Biot's equations.

This paper is organized as follows. In \cref{s:biot,s:biotdiscrete} we
present Biot's consolidation model and its HDG discretization. Uniform
well-posedness of the discretization is proven in \cref{s:uniwp}. This
result is used in \cref{s:precon} to determine parameter-robust
preconditioners. Numerical examples demonstrating parameter-robustness
of the preconditioner are given in \cref{s:numex} and conclusions are
drawn in \cref{s:conclusions}.

\section{Biot's Consolidation Model}
\label{s:biot}

We consider the following formulation of Biot's consolidation model
\cite{lee2017parameter}:
\begin{subequations}
  \label{eq:biot}
  \begin{align}
    -\nabla \cdot 2 \tilde{\mu} \varepsilon(u) + \nabla p_T
    &= \tilde{f} && \text{ in } \Omega \times I,
    \\
    -\nabla \cdot u + \lambda^{-1}(\alpha p - p_T)
    &= 0 && \text{ in } \Omega \times I,
    \\
    c_0\partial_tp + \alpha\lambda^{-1}(\alpha\partial_t p - \partial_tp_T) + \nabla \cdot \tilde{z}
    &= \tilde{g} && \text{ in } \Omega \times I,
    \\
    \tilde{\kappa}^{-1}\tilde{z} + \nabla p
    &= 0 && \text{ in } \Omega \times I,
  \end{align}
\end{subequations}
in which $u$ is the displacement, $p$ is the fluid pressure of the
fluid, $p_T$ is the total pressure, $\tilde{z}$ is the Darcy velocity,
$\tilde{f}$ is a given body force, $\tilde{g}$ is a given source/sink
term, $\tilde{\kappa} > 0$ a constant representing the permeability of
the porous media, $c_0 \ge 0$ is the specific storage coefficient,
$\alpha \in (0,1)$ is the Biot--Willis constant, and $\lambda$ and
$\tilde{\mu}$ are the Lam\'e constants. In what follows, we will use
that $2\tilde{\mu}\lambda^{-1} \le c_l$, where $c_l > 0$ is a
constant. Furthermore, $I=(0,T]$ is the time-interval with final time
$T$, $\Omega \subset \mathbb{R}^d$ ($d=2,3$) is an open and bounded
polygonal ($d=2$) or polyhedral ($d=3$) domain, and
$\varepsilon(u) := (\nabla u + (\nabla u)^T)/2$.

Discretizing Biot's consolidation model \cref{eq:biot} in time using
backward Euler time-stepping with time-step $\tau$, and defining
$z = \tau \tilde{z}$, we require to solve at each time level
$t^n = n\tau$ ($n=1,\hdots,N$ where $N=T/\tau$) the problem
\begin{subequations}
  \label{eq:biotdt}
  \begin{align}
    -\nabla \cdot \mu\varepsilon(u) + \nabla p_T
    &= f && \text{ in } \Omega,
    \\
    -\nabla \cdot u + \lambda^{-1}(\alpha p - p_T)
    &= 0 && \text{ in } \Omega,
    \\
    c_0p + \alpha\lambda^{-1}(\alpha p - p_T) + \nabla \cdot z
    &= g && \text{ in } \Omega,
    \\
    \kappa^{-1}z + \nabla p
    &= 0 && \text{ in } \Omega,
  \end{align}
\end{subequations}
where $f$ and $g$ are suitable modifications of $\tilde{f}$ and
$\tilde{g}$ and, to simplify notation in the remainder of this
manuscript, we have defined $\mu := 2\tilde{\mu}$ and
$\kappa := \tau\tilde{\kappa}$. To close the system, we prescribe
$u=0$ and $p=0$ on $\partial \Omega$. In what follows we will consider
preconditioners for a hybridizable discontinuous Galerkin
discretization of \cref{eq:biotdt}.

\section{The Hybridizable Discontinuous Galerkin Discretization}
\label{s:biotdiscrete}

We consider a triangulation $\mathcal{T}_h := \cbr[0]{K}$ consisting
of non-overlapping simplices $K$ such that
$\overline{\Omega} = \cup_{K \in \mathcal{T}_h} K$. We denote by $h_K$
the characteristic length of a cell $K$ and by $n$ the outward unit
normal vector on $\partial K$. The set of all faces is denoted by
$\mathcal{F}_h$. The union of all faces is denoted by $\Gamma_0$.

The `cell' function spaces are defined as
\begin{align*}
    V_h
    &:= \cbr[1]{v_h\in \sbr[0]{L^2(\Omega)}^d
      : \ v_h \in \sbr[0]{\mathbb{P}_k(K)}^d, \ \forall\ K\in\mathcal{T}_h},
    \\
    Q_h
    &:= \cbr[1]{q_h\in L^2(\Omega) : \ q_h \in \mathbb{P}_{k-1}(K) ,\
      \forall \ K \in \mathcal{T}_h},
\end{align*}
while the `face' function spaces are defined as
\begin{align*}
  \bar{V}_h
  &:=
    \cbr[1]{\bar{v}_h \in \sbr[0]{L^2(\Gamma_0)}^d:\ \bar{v}_h \in
    \sbr[0]{\mathbb{P}_{k}(F)}^d\ \forall\ F \in \mathcal{F}_h,\ \bar{v}_h
    = 0 \text{ on } \partial \Omega}.
  \\
  \bar{Q}_h
  &:= \cbr[1]{\bar{q}_h \in L^2(\Gamma_0) : \ \bar{q}_h \in
    \mathbb{P}_{k}(F) \ \forall\ F \in \mathcal{F}_h},
  \\
  \bar{Q}_h^0
  &:=\cbr[0]{\bar{\psi}_h \in \bar{Q}_h:
    \bar{\psi}_h=0 \text{ on } \partial \Omega}.
\end{align*}
Function space pairs are denoted in boldface, i.e.,
$\boldsymbol{V}_h := V_h \times \bar{V}_h$,
$\boldsymbol{Q}_h := Q_h \times \bar{Q}_h$, and
$\boldsymbol{Q}_h^0 := Q_h \times \bar{Q}_h^0$. We also use boldface
for the cell/facet function pair, e.g.,
$\boldsymbol{v}_h := (v_h, \bar{v}_h) \in \boldsymbol{V}_h$, and
similarly for cell/facet function pairs in $\boldsymbol{Q}_h$ and
$\boldsymbol{Q}_h^0$. Finally, we define
$\boldsymbol{X}_h := \boldsymbol{V}_h \times \boldsymbol{Q}_h \times
V_h \times \boldsymbol{Q}_h^0$.

We use the following inner product notation:
$(p,q)_{\mathcal{T}_h} := \sum_{K \in \mathcal{T}_h} (p,q)_K$ and
$\langle p,q \rangle_{\partial \mathcal{T}_h} := \sum_{K \in
  \mathcal{T}_h} \langle p,q \rangle_{\partial K}$, where
$(p,q)_K = \int_K pq \dif x$ and
$\langle p,q \rangle_{\partial K} = \int_{\partial K} pq \dif s$ if
$p,q$ are scalar. Similar definitions hold if $p,q$ are
vectors. Furthermore, we define the following forms:
\begin{subequations}
  \begin{align}
    \label{eq:d_h}
    d_h(\boldsymbol{u}_h, \boldsymbol{v}_h)
    :=& (\mu \varepsilon(u_h), \varepsilon(v_h))_{\mathcal{T}_h}
        +\langle\tfrac{\eta\mu}{h_K}
        (u_h - \bar{u}_h),v_h-\bar{v}_h\rangle_{\partial \mathcal{T}_h}
    \\ \nonumber
      &- \langle \mu \varepsilon(u_h)n,
        v_h-\bar{v}_h\rangle_{\partial \mathcal{T}_h}
        - \langle \mu \varepsilon(v_h)n,
        u_h-\bar{u}_h\rangle_{\partial \mathcal{T}_h},
    \\
    \label{eq:b_h}
    b_h(v_h, \boldsymbol{q}_h)
    :=& - (q_h, \nabla\cdot v_h )_{\mathcal{T}_h}
        + \langle \bar{q}_h, v_h \cdot n\rangle_{\partial \mathcal{T}_h},
  \end{align}
\end{subequations}
where $\eta > 1$ is the interior penalty parameter.

The HDG discretization of the Biot equations \cref{eq:biotdt} is given
by \cite{cesmelioglu2023analysis}: Find
$\boldsymbol{x}_h := (\boldsymbol{u}_h,\boldsymbol{p}_{Th}, z_h,
\boldsymbol{p}_h) \in \boldsymbol{X}_h$, such that
\begin{equation}
  \label{eq:biotcomp}
  a_h(\boldsymbol{x}_h, \boldsymbol{y}_h) = (f,v_h)_{\mathcal{T}_h} - (g,q_h)_{\mathcal{T}_h}
  \quad \forall \boldsymbol{y}_h := (\boldsymbol{v}_h,\boldsymbol{q}_{Th}, w_h,
  \boldsymbol{q}_h) \in \boldsymbol{X}_h,
\end{equation}
where
\begin{equation}
  \label{eq:definitionah}
  \begin{split}
    a_h(\boldsymbol{x}_h, \boldsymbol{y}_h)
    =& d_h(\boldsymbol{u}_h, \boldsymbol{v}_h)
    + b_h(v_h,\boldsymbol{p}_{Th})
    + b_h(u_h, \boldsymbol{q}_{Th})     
    \\
    & + \del[1]{\kappa^{-1}z_h, w_h}_{\mathcal{T}_h}     
    + b_h(w_h, \boldsymbol{p}_h)
    + b_h(z_h, \boldsymbol{q}_h)
    \\
    & - (c_0 p_h, q_h)_{\mathcal{T}_h}
    - \lambda^{-1}\del[1]{\alpha p_{h}-p_{Th}, \alpha q_h - q_{Th}}_{\mathcal{T}_h}.
  \end{split}
\end{equation}

\section{Uniform Well-Posedness}
\label{s:uniwp}

In this section we prove well-posedness of the HDG method
\cref{eq:biotcomp}.

\subsection{Inner Products and Norms}
\label{ss:inprod}

We will require the following inner products on $\boldsymbol{V}_h$,
$\boldsymbol{Q}_h$, $V_h$, and $\boldsymbol{Q}_h^0$:
\begin{subequations}
  \label{eq:inprods}
  \begin{align}
    (\boldsymbol{u}_h,\boldsymbol{v}_h)_{v}
    :=& \mu (\varepsilon(u_h), \varepsilon(v_h))_{\mathcal{T}_h}
        + \mu \eta \langle h_K^{-1}(u_h-\bar{u}_h), v_h - \bar{v}_h \rangle_{\partial \mathcal{T}_h},
    \\
    (\boldsymbol{p}_{Th},\boldsymbol{q}_{Th})_{q_T}
    :=& \mu^{-1}(p_{Th}, q_{Th})_{\mathcal{T}_h}
        + \mu^{-1} \eta^{-1} \langle h_K \bar{p}_{Th}, \bar{q}_{Th} \rangle_{\partial \mathcal{T}_h},
    \\
    (z_h,w_h)_{w}
    :=& \kappa^{-1}(z_h, w_h)_{\mathcal{T}_h},
    \\
    (\boldsymbol{p}_h, \boldsymbol{q}_h)_{q}
    :=& (c_0 + \alpha^2 \lambda^{-1}) (p_h,q_h)_{\mathcal{T}_h}
        +  \kappa (\nabla p_h, \nabla q_h)_{\mathcal{T}_h}
        + \kappa \eta \langle h_K^{-1}(p_h-\bar{p}_h), q_h-\bar{q}_h \rangle_{\partial \mathcal{T}_h}.
  \end{align}  
\end{subequations}
Note that we use a discrete $H^1$ type norm for
$\boldsymbol{p}_h$. This is because a preconditioner based on the
standard $L^2$ norm fails to give parameter-robust
preconditioning. This was verified numerically for the Darcy equations
in \cite{henriquez2025parameter}. The norms induced by these inner
products are denoted as $\tnorm{\boldsymbol{v}_h}_v$,
$\tnorm{\boldsymbol{q}_{Th}}_{q_T}$, $\tnorm{w_h}_w$, and
$\tnorm{\boldsymbol{q}_h}_q$. On $\boldsymbol{X}_h$ we define the
inner product
\begin{equation}
  \label{eq:inprodXh}
  ((\boldsymbol{u}_h,\boldsymbol{p}_{Th},z_h,\boldsymbol{p}_h),
  (\boldsymbol{v}_h,\boldsymbol{q}_{Th},w_h,\boldsymbol{q}_h))_{\boldsymbol{X}_h}  
  := (\boldsymbol{u}_h,\boldsymbol{v}_h)_{v} + (\boldsymbol{p}_{Th},\boldsymbol{q}_{Th})_{q_T}
  + (z_h,w_h)_{w} + (\boldsymbol{p}_h, \boldsymbol{q}_h)_{q},
\end{equation}
and its induced norm:
\begin{equation*}
  \tnorm{(\boldsymbol{v}_h,\boldsymbol{q}_{Th},w_h,\boldsymbol{q}_h)}_{\boldsymbol{X}_h}^2
  \\
  :=
  \tnorm{\boldsymbol{v}_h}_v^2
  + \tnorm{\boldsymbol{q}_{Th}}_{q_T}^2
  + \tnorm{w_h}_w^2
  + \tnorm{\boldsymbol{q}_h}_q^2.
\end{equation*}
Furthermore, we will require the norm
\begin{equation*}
  \tnorm{\boldsymbol{q}_h}_{1,p}^2
  := \norm[0]{\nabla q_h}_{\mathcal{T}_h}^2
  + \eta \norm[0]{h_K^{-1/2}(q_h - \bar{q}_h)}_{\partial \mathcal{T}_h}^2
  \quad \forall \boldsymbol{q}_h \in \boldsymbol{Q}_h^0.
\end{equation*}

\subsection{Boundedness and Stability}
\label{ss:bands}

In what follows, we will refer to a constant as being a uniform
constant if it is independent of the model parameters ($\mu>0$,
$\lambda\ge c_l^{-1} \mu$, $0<\alpha \le 1$, $\kappa>0$, $c_0 \ge 0$)
and the mesh size $h$. We first present some preliminary results. By
\cite[Lemma 4.3]{rhebergen2017analysis} and definition of the norms,
there exists a uniform constant $c_1>0$ such that
\begin{equation}
  \label{eq:dhbound}
  |d_h(\boldsymbol{u}_h, \boldsymbol{v}_h)|
  \le c_1 \tnorm{\boldsymbol{u}_h}_v \tnorm{\boldsymbol{v}_h}_v
  \quad
  \forall \boldsymbol{u}_h, \boldsymbol{v}_h \in \boldsymbol{V}_h.
\end{equation}
By \cite[Lemma 2]{cesmelioglu2020embedded} and \cite[Lemma
4.2]{rhebergen2017analysis}, there exists a uniform constant $c_d > 0$
such that
\begin{equation}
  \label{eq:stabdh}
  d_h(\boldsymbol{v}_h, \boldsymbol{v}_h)
  \ge c_d \tnorm{\boldsymbol{v}_h}_v^2
  \quad
  \forall \boldsymbol{v}_h \in \boldsymbol{V}_h.
\end{equation}
It is shown in \cite[Lemma 4]{kraus2021uniformly} that there exists a
uniform constant $c_2>0$ such that
\begin{equation}
  \label{eq:stabbhQh0}
  \sup_{0 \ne w_h \in V_h}
  \frac{b_h(w_h, \boldsymbol{q}_h)}{\tnorm{w_h}_w }
  \ge c_2 \kappa^{1/2}\tnorm{\boldsymbol{q}_h}_{1,p}
  \quad \forall \boldsymbol{q}_h \in \boldsymbol{Q}_h^0,    
\end{equation}
while in \cite[Lemma 1]{rhebergen2018preconditioning} and \cite[Lemma
8]{rhebergen2020embedded} it is shown that there exists a uniform
constant $c_3>0$ such that
\begin{equation}
  \label{eq:stabbhQh}
  \sup_{\boldsymbol{0} \ne \boldsymbol{v}_h \in \boldsymbol{V}_h}
  \frac{b_h(v_h, \boldsymbol{q}_{Th})}{\tnorm{\boldsymbol{v}_h}_{v} }
  \ge c_3 \tnorm{\boldsymbol{q}_{Th}}_{q_T}
  \quad \forall \boldsymbol{q}_{Th} \in \boldsymbol{Q}_h.
\end{equation}

\begin{lemma}[Boundedness]
  \label{lem:bounded}
  There exists a uniform constant $c_b > 0$ such that
  \begin{equation*}
    |a_h(\boldsymbol{x}_h, \boldsymbol{y}_h)|
    \le c_b \tnorm{\boldsymbol{x}_h}_{\boldsymbol{X}_h} \tnorm{\boldsymbol{y}_h}_{\boldsymbol{X}_h}
    \qquad
    \forall \boldsymbol{x}_h,\boldsymbol{y}_h \in \boldsymbol{X}_h.
  \end{equation*}
\end{lemma}
\begin{proof}
  By the Cauchy--Schwarz inequality, H\"older's inequality for sums,
  the single-valuedness of $\bar{v}_h$ on interior faces, and that
  $\bar{v}_h = 0$ on $\partial \Omega$, we find that there exists a
  uniform constant $c_1^b>0$ such that
  \begin{equation}
    \label{eq:bhboundedness1}
    |b_h(v_h,\boldsymbol{q}_{Th})|
    \le c_1^b \tnorm{\boldsymbol{v}_h}_v \tnorm{\boldsymbol{q}_{Th}}_{q_T}
    \quad
    \forall (\boldsymbol{v}_h,\boldsymbol{q}_{Th}) \in \boldsymbol{V}_h \times \boldsymbol{Q}_h.
  \end{equation}
  Applying integration by parts, the Cauchy--Schwarz inequality, and a
  discrete trace inequality \cite[Lemma 1.46]{di2011mathematical}, we
  find that there exists a uniform constant $c_2^b > 0$ such that
  \begin{equation*}
    |b_h(w_h,\boldsymbol{q}_h)|
    \le c_2^b \tnorm{w_h}_w \tnorm{\boldsymbol{q}_h}_q
    \quad
    \forall (w_h,\boldsymbol{q}_h) \in V_h \times \boldsymbol{Q}_h^0.
  \end{equation*}
  By the Cauchy--Schwarz inequality we find
  \begin{align*}
    \kappa^{-1}|(z_h, w_h)_{\mathcal{T}_h}| &\le \tnorm{z_h}_w \tnorm{w_h}_w
    &&
       \forall z_h,w_h \in V_h,
    \\
    c_0|(p_h,q_h)_{\mathcal{T}_h}| &\le \tnorm{\boldsymbol{p}_h}_q \tnorm{\boldsymbol{q}_h}_q
    &&
       \forall p_h,q_h \in Q_h.      
  \end{align*}
  Consider now the last term in the definition of $a_h(\cdot, \cdot)$,
  see \cref{eq:definitionah}. We have, using the Cauchy--Schwarz
  inequality and $\lambda^{-1}\mu \le c_l$,
  \begin{equation*}
    \begin{split}
      \lambda^{-1}|(\alpha p_h - p_{Th}, q_{Th})|
      &\le
      c_l^{1/2} \tnorm{\boldsymbol{p}_h}_q \tnorm{\boldsymbol{q}_{Th}}_{q_T}
      + c_l \tnorm{\boldsymbol{p}_{Th}}_{q_T} \tnorm{\boldsymbol{q}_{Th}}_{q_T},
      \\
      \lambda^{-1}|(\alpha p_h - p_{Th}, \alpha q_h)|
      &\le \tnorm{\boldsymbol{p}_h}_{q} \tnorm{\boldsymbol{q}_h}_{q}
      + c_l^{1/2}\tnorm{\boldsymbol{p}_{Th}}_{q_T} \tnorm{\boldsymbol{q}_h}_{q}.
    \end{split}
  \end{equation*}
  Combining the above inequalities with \cref{eq:dhbound}, and using
  the discrete version of H\"older's inequality, the result
  follows.
\end{proof}

\begin{lemma}[Stability]
  \label{lem:stability}
  There exists a uniform constant $c_s > 0$ such that
  \begin{equation}
    \label{eq:infsup}
    c_s \le \inf_{\boldsymbol{x}_h \in \boldsymbol{X}_h} \sup_{\boldsymbol{y}_h \in \boldsymbol{X}_h}
    \frac{a_h(\boldsymbol{x}_h, \boldsymbol{y}_h)}{\tnorm{\boldsymbol{x}_h}_{\boldsymbol{X}_h} \tnorm{\boldsymbol{y}_h}_{\boldsymbol{X}_h}}.
  \end{equation}
\end{lemma}
\begin{proof}
  To prove \cref{eq:infsup} we will show that for all
  $\boldsymbol{x}_h \in \boldsymbol{X}_h$ there exists a
  $\boldsymbol{y}_h = \boldsymbol{y}_h(\boldsymbol{x}_h) \in
  \boldsymbol{X}_h$ and positive uniform constants $c_{as},c_b$ such
  that
  \begin{subequations}
    \begin{align}
      \label{eq:tnormxXhcs}
      a_h(\boldsymbol{x}_h, \boldsymbol{y}_h) & \ge c_{as} \tnorm{\boldsymbol{x}_h}_{\boldsymbol{X}_h}^2,
      \\
      \label{eq:tnormyXhcb}
      \tnorm{\boldsymbol{y}_h}_{\boldsymbol{X}_h} & \le c_b \tnorm{\boldsymbol{x}_h}_{\boldsymbol{X}_h}.      
    \end{align}
  \end{subequations}
  A consequence of \cref{eq:stabbhQh0,eq:stabbhQh} is that for given
  $\boldsymbol{p}_h \in \boldsymbol{Q}_h^0$ and
  $\boldsymbol{p}_{Th} \in \boldsymbol{Q}_h$ there exist
  $\tilde{z}_h \in V_h$ and
  $\tilde{\boldsymbol{u}}_h \in \boldsymbol{V}_h$ such that
  \begin{subequations}
    \label{eq:infsupresults}
    \begin{align}
      \label{eq:bhztilde}
      b_h(\tilde{z}_h,\boldsymbol{p}_h) = \kappa \tnorm{\boldsymbol{p}_h}_{1,p}^2
      &\quad\text{ and } \quad \tnorm{\tilde{z}_h}_w \le c_2^{-1}\kappa^{1/2} \tnorm{\boldsymbol{p}_h}_{1,p},
      \\
      \label{eq:bhutilde}
      b_h(\tilde{u}_h,\boldsymbol{p}_{Th}) = \tnorm{\boldsymbol{p}_{Th}}_{q_T}^2
      &\quad\text{ and } \quad \tnorm{\tilde{\boldsymbol{u}}_h}_v \le c_3^{-1} \tnorm{\boldsymbol{p}_{Th}}_{q_T}.
    \end{align}
  \end{subequations}
  Given a non-null
  $\boldsymbol{x}_h :=
  (\boldsymbol{u}_h,\boldsymbol{p}_{Th},z_h,\boldsymbol{p}_h) \in
  \boldsymbol{X}_h$, we define
  $\boldsymbol{y}_h :=
  (\boldsymbol{v}_h,\boldsymbol{q}_{Th},w_h,\boldsymbol{q}_h) \in
  \boldsymbol{X}_h$ such that
  \begin{equation*}
    \boldsymbol{v}_h = \boldsymbol{u}_h + \delta_1 \tilde{\boldsymbol{u}}_h,
    \quad
    \boldsymbol{q}_{Th} = - \boldsymbol{p}_{Th},
    \quad
    w_h = z_h + \delta_2 \tilde{z}_h,
    \quad\text{and}\quad
    \boldsymbol{q}_h = -\boldsymbol{p}_h,
  \end{equation*}
  where $\delta_1,\delta_2$ are positive constants to be determined.

  We first prove \cref{eq:tnormxXhcs}. Using
  \cref{eq:dhbound,eq:infsupresults,eq:stabdh}, the Cauchy--Schwarz
  inequality, and Young's inequality with constant $\delta_3 > 0$, we
  find that:
  \begin{align*}
    a_h(\boldsymbol{x}_h, \boldsymbol{y}_h)
    =& d_h(\boldsymbol{u}_h, \boldsymbol{u}_h)
       + \delta_1 d_h(\boldsymbol{u}_h, \tilde{\boldsymbol{u}}_h)
       + \delta_1 b_h(\tilde{u}_h,\boldsymbol{p}_{Th})
       + \del[1]{\kappa^{-1}z_h, z_h}_{\mathcal{T}_h}
       + \delta_2 \del[1]{\kappa^{-1}z_h, \tilde{z}_h}_{\mathcal{T}_h}
       + \delta_2 b_h(\tilde{z}_h, \boldsymbol{p}_h)
    \\
     & + (c_0 p_h, p_h)_{\mathcal{T}_h}
       + \lambda^{-1}\del[1]{\alpha p_{h}-p_{Th}, \alpha p_h - p_{Th}}_{\mathcal{T}_h}
    \\
    \ge& (c_d - \tfrac{\delta_1}{2} c_1^2c_3^{-2}) \tnorm{\boldsymbol{u}_h}_v^2
         + \tfrac{\delta_1}{2}\tnorm{\boldsymbol{p}_{Th}}_{q_T}^2
         + (1 - \tfrac{\delta_2}{2}c_2^{-2})\tnorm{z_h}_w^2
         + \tfrac{\delta_2}{2} \kappa \tnorm{\boldsymbol{p}_h}_{1,p}^2
         + c_0\norm[0]{p_h}_{\mathcal{T}_h}^2
    \\
     &  + (1-\delta_3)\lambda^{-1}\alpha^2\norm[0]{p_h}_{\mathcal{T}_h}^2
       + \lambda^{-1}\norm[0]{p_{Th}}_{\mathcal{T}_h}^2
       - \tfrac{1}{4}\delta_3^{-1} \lambda^{-1} \norm[0]{p_{Th}}_{\mathcal{T}_h}^2.
  \end{align*}
  Choosing $\delta_1 = c_dc_3^2c_1^{-2}$, $\delta_2 = c_2^2$,
  $\delta_3 = 1/(4+c_dc_3^2c_1^{-2}c_l^{-1})$, and using $\lambda^{-1}\mu \le c_l$, we find
  that
  \begin{align*}
    a_h(\boldsymbol{x}_h, \boldsymbol{y}_h)
    \ge & \tfrac{c_d}{2} \tnorm{\boldsymbol{u}_h}_v^2
          + \tfrac{c_dc_3^2}{4c_1^2}\tnorm{\boldsymbol{p}_{Th}}_{q_T}^2
          + \tfrac{1}{2}\tnorm{z_h}_w^2
          + \tfrac{c_2^2}{2} \kappa \tnorm{\boldsymbol{p}_h}_{1,p}^2
          + c_0\norm[0]{p_h}_{\mathcal{T}_h}^2
    \\
        & + (1-1/(4+c_dc_3^2c_1^{-2}c_l^{-1}))\lambda^{-1}\alpha^2\norm[0]{p_h}_{\mathcal{T}_h}^2.
  \end{align*}
  Defining
  $c_s = \tfrac{1}{2}\min(c_d, \tfrac{1}{2}c_dc_3^2c_1^{-2}, 1,
  c_2^2, 2(1-1/(4+c_dc_3^2c_1^{-2}c_l^{-1})))$ we obtain
  \cref{eq:tnormxXhcs}.

  We now prove \cref{eq:tnormyXhcb}. Using \cref{eq:bhutilde},
  \begin{equation}
    \label{eq:vnormvhuut}
    \tnorm{\boldsymbol{v}_h}_v^2
    \le 2 \tnorm{\boldsymbol{u}_h}_v^2 + 2\delta_1^2\tnorm{\tilde{\boldsymbol{u}}_h}_v^2
    \le 2 \tnorm{\boldsymbol{u}_h}_v^2 + 2\delta_1^2c_3^{-2} \tnorm{\boldsymbol{p}_{Th}}_{q_T}^2,
  \end{equation}
  and using \cref{eq:bhztilde},
  \begin{equation}
    \label{eq:wnormwhzzt}
    \begin{split}
      \tnorm{w_h}_w^2
      \le 2 \tnorm{z_h}_w^2 + 2 \delta_2^2 \tnorm{\tilde{z}_h}_w^2
      &\le 2 \tnorm{z_h}_w^2 + 2 \delta_2^2 c_2^{-2} \kappa \tnorm{\boldsymbol{p}_h}_{1,p}^2
      \\
      &\le 2 \tnorm{z_h}_w^2 + 2 \delta_2^2 c_2^{-2} \tnorm{\boldsymbol{p}_h}_{q}^2.
    \end{split}
  \end{equation}
  \Cref{eq:tnormyXhcb} follows by combining
  \cref{eq:vnormvhuut,eq:wnormwhzzt} and noting that
  $\tnorm{\boldsymbol{q}_{Th}}_{q_T}=\tnorm{\boldsymbol{p}_{Th}}_{q_T}$
  and
  $\tnorm{\boldsymbol{q}_{h}}_{q}=\tnorm{\boldsymbol{p}_{h}}_{q}$.
\end{proof}

An immediate consequence of \cref{lem:bounded,lem:stability} is the
uniform well-posedness of the HDG method \cref{eq:biotcomp}.

\section{Preconditioning}
\label{s:precon}

Let $\boldsymbol{X}_h^{\ast}$ be the dual space of $\boldsymbol{X}_h$
and denote by
$\langle \cdot, \cdot
\rangle_{\boldsymbol{X}_h^{\ast},\boldsymbol{X}_h}$ the duality
pairing between $\boldsymbol{X}_h$ and its dual. The discrete Biot
problem \cref{eq:biotcomp} is then equivalent to finding
$\boldsymbol{x}_h \in \boldsymbol{X}_h$ such that
\begin{equation}
  \label{eq:biotcompA}
  A \boldsymbol{x}_h = \boldsymbol{f}_h \quad \text{in} \quad \boldsymbol{X}_h^{\ast},
\end{equation}
where $A : \boldsymbol{X}_h \to \boldsymbol{X}_h^{\ast}$ is defined by
$\langle A \boldsymbol{x}_h, \boldsymbol{y}_h
\rangle_{\boldsymbol{X}_h^{\ast},\boldsymbol{X}_h} =
a_h(\boldsymbol{x}_h, \boldsymbol{y}_h)$ for all
$\boldsymbol{x}_h,\boldsymbol{y}_h \in \boldsymbol{X}_h$. Since
\cref{lem:bounded,lem:stability} hold with uniform constants $c_b$ and
$c_s$, it is known that the preconditioner
$P^{-1} : \boldsymbol{X}_h^{\ast} \to \boldsymbol{X}_h$ defined by the
inner product $(\cdot, \cdot)_{\boldsymbol{X}_h}$, i.e.,
$\langle P \boldsymbol{x}_h, \boldsymbol{y}_h
\rangle_{\boldsymbol{X}_h^{\ast},\boldsymbol{X}_h} =
(\boldsymbol{x}_h,\boldsymbol{y}_h)_{\boldsymbol{X}_h}$, is
parameter-robust, i.e., the condition number of the preconditioned
system is independent of the parameters $\mu > 0$, $\lambda > 0$ (with
$\mu \lambda^{-1} \le c_l$), $\kappa > 0$, $0<\alpha\le 1$, $c_0\ge 0$
(see \cite{mardal2011preconditioning,lee2017parameter}). For an HDG
discretization, however, it is preferred to solve the statically
condensed form of \cref{eq:biotcompA}. In this section we therefore
determine a parameter-robust preconditioner for this reduced form.

\subsection{Preconditioning the Statically Condensed Problem}
\label{ss:preconSC}

We briefly summarize the general framework presented in
\cite{henriquez2025parameter} to determine a parameter-robust
preconditioner for the statically condensed form of \cref{eq:biotcomp}
given the preconditioner $P$. Let
$X_h := V_h \times Q_h \times V_h \times Q_h$ denote the space of all
cell-based functions and
$\bar{X}_h := \bar{V}_h \times \bar{Q}_h \times \bar{Q}_h^0$ the space
of all face-based functions. Let
$x_h := (u_h,p_{Th},z_h,p_h) \in X_h$,
$\bar{x}_h := (\bar{u}_h,\bar{p}_{Th},\bar{p}_h) \in \bar{X}_h$,
$\boldsymbol{x}_h=(x_h,\bar{x}_h)$, and write \cref{eq:biotcompA} in
block form as
\begin{equation}
  \label{eq:blockA}
  \begin{bmatrix}
    A_{11} & A_{21}^T
    \\
    A_{21} & A_{22}
  \end{bmatrix}
  \begin{bmatrix}
    x_h \\ \bar{x}_h
  \end{bmatrix}
  =
  \begin{bmatrix}
    f_h \\ 0
  \end{bmatrix}.
\end{equation}
Here $f_h$ corresponds to the discrete form of $(f, 0, g, 0)$ in
\cref{eq:biotdt}. We can also write the preconditioner $P$ in block
form:
\begin{equation*}
  P =
  \begin{bmatrix}
    P_{11} & P_{21}^T
    \\
    P_{21} & P_{22}
  \end{bmatrix},
\end{equation*}
where $P_{11} : X_h \to \bar{X}_h^{\ast}$,
$P_{21} : X_h \to \bar{X}_h^{\ast}$, and
$P_{22} : \bar{X}_h \to \bar{X}_h^{\ast}$.

Eliminating cell-based functions, the reduced form of \cref{eq:blockA}
is given by
\begin{equation}
  \label{eq:SAbiot}
  S_A \bar{x}_h = \bar{b}_h,
\end{equation}
where $S_A := A_{22} - A_{21}A_{11}^{-1}A_{21}^T$ is the Schur
complement of $A$ and $\bar{b}_h := -
A_{21}A_{11}^{-1}f_h$. The Schur complement
$S_P : \bar{X}_h \to \bar{X}_h^{\ast}$ of $P$ is defined as
$S_P := P_{22} - P_{21}P_{22}^{-1}P_{21}^T$ and it defines an inner
product on $\bar{X}_h$:
\begin{equation*}
  \langle S_P \bar{x}_h, \bar{y}_h \rangle_{\bar{X}_h^{\ast},\bar{X}_h}
  = (\bar{x}_h, \bar{y}_h)_{\bar{X}_h}
  \quad \forall \bar{x}_h,\bar{y}_h \in \bar{X}_h.
\end{equation*}
The norm induced by this inner product is denoted by
$\norm[0]{\cdot}_{\bar{X}_h}$. In \cite[Theorem
1]{henriquez2025parameter} we showed that $S_P$ is a
parameter-robust preconditioner for \cref{eq:SAbiot} provided there
exists a uniform constant $c_p > 0$ such that
\begin{equation}
  \label{eq:thm23a}
  \tnorm{(-A_{11}^{-1}A_{21}^T\bar{x}_h, \bar{x}_h)}_{\boldsymbol{X}_h} \le c_p \norm[0]{\bar{x}_h}_{\bar{X}_h}
  \quad \forall \bar{x}_h \in \bar{X}_h.
\end{equation}
In the remainder of this section we prove the existence of the uniform
constant $c_p$.

\subsection{Local Solvers}
\label{ss:prelims}

In this section we define the HDG local solvers and present some
useful results.

\subsubsection{Local Solvers for Biot's model}
\label{sss:localsolversBiot}

Let $t_h := (u_h,p_{Th},z_h,p_h) \in X_h$,
$\bar{t}_h := (\bar{m}_h,\bar{r}_{Th},\bar{r}_h) \in \bar{X}_h$,
$y_h := (v_h,q_{Th},w_h,q_h) \in X_h$, and
$\bar{y}_h := (\bar{v}_h,\bar{q}_{Th},\bar{q}_h) \in
\bar{X}_h$. Furthermore, let $s \in L^2(\Omega)^d$ and
$\tilde{s} \in L^2(\Omega)$. We define the local bilinear and linear
forms for each cell $K \in \mathcal{T}_h$ as:
\begin{equation}
  \label{eq:localah}
  \begin{split}
    a_K(t_h,y_h)
    &:= a_h((t_h,0),(y_h,0))|_K,
    \\
    L_K(y_h)
    &:= (s,v_h)_K - (\tilde{s},q_h)_K - a_h((0,\bar{t}_h),(y_h,0))|_K,
  \end{split}
\end{equation}
where by $a_h(\cdot, \cdot)|_K$ we mean the restriction of
$a_h(\cdot, \cdot)$ to cell $K \in \mathcal{T}_h$. The functions
$u_h^L(\bar{t}_h,s,\tilde{s}) \in V_h$,
$p_{Th}^L(\bar{t}_h,s,\tilde{s}) \in Q_h$,
$z_h^L(\bar{t}_h,s,\tilde{s}) \in V_h$,
$p_h^L(\bar{t}_h,s,\tilde{s}) \in Q_h$ are defined to be such that
their restriction to cell $K$ satisfies the following local problem
for given $\bar{t}_h \in \bar{X}_h$, $s \in L^2(\Omega)^d$, and
$\tilde{s} \in L^2(\Omega)$:
\begin{equation}
  \label{eq:localproblembiot}
  a_K(t_h^L,y_h) = L_K(y_h) \quad \forall y_h \in V(K) \times Q(K) \times V(K) \times Q(K),
\end{equation}
where $t_h^L := (u_h^L,p_{Th}^L,z_h^L,p_h^L)$,
$V(K) := \sbr[0]{\mathbb{P}_k(K)}^d$, and
$Q(K) := \mathbb{P}_{k-1}(K)$.

The local solution to \cref{eq:localproblembiot} is now used to
eliminate $x_h = (u_h,p_{Th},z_h,p_h)$ from \cref{eq:biotcomp}. We
summarize this in the following lemma, the proof of which is omitted
as the steps are identical to those used to prove \cite[Lemma
4]{rhebergen2018preconditioning}.

\begin{lemma}
  \label{lem:redvarform}
  Define $u_h^f := u_h^L(0,f,g)$, $p_{Th}^f := p_{Th}^L(0,f,g)$,
  $z_h^f := z_h^L(0,f,g)$, and $p_h^f := p_h^L(0,f,g)$ for given
  $f \in L^2(\Omega)^d$, and $g \in L^2(\Omega)$. Furthermore, for all
  $\bar{y}_h := (\bar{v}_h,\bar{q}_{Th},\bar{q}_h) \in \bar{X}_h$
  define $l_u(\bar{y}_h) := u_h^L(\bar{y}_h,0,0)$,
  $l_{p_T}(\bar{y}_h) := p_{Th}^L(\bar{y}_h,0,0)$,
  $l_z(\bar{y}_h) := z_h^L(\bar{y}_h,0,0)$, and
  $l_p(\bar{y}_h) := p_h^L(\bar{y}_h,0,0)$. Let
  $\bar{x}_h := (\bar{u}_h,\bar{p}_{Th},\bar{p}_h) \in \bar{X}_h$
  solve
  \begin{equation*}
    \bar{a}_h(\bar{x}_h, \bar{y}_h) = (f, l_u(\bar{y}_h))_{\mathcal{T}_h} - (g, l_p(\bar{y}_h))_{\mathcal{T}_h}
    \quad \forall \bar{y}_h \in \bar{X}_h,
  \end{equation*}
  where
  \begin{align*}
    \bar{a}_h(\bar{x}_h, \bar{y}_h)
    =& d_h((l_u(\bar{x}_h),\bar{u}_h),(l_u(\bar{y}_h),\bar{v}_h))
       + \del[1]{\kappa^{-1}l_z(\bar{x}_h), l_z(\bar{y}_h)}_{\mathcal{T}_h}     
    \\
     & + b_h(l_u(\bar{y}_h), (l_{p_T}(\bar{x}_h),\bar{p}_{Th}))
       + b_h(l_u(\bar{x}_h), (l_{p_T}(\bar{y}_h),\bar{q}_{Th}))
    \\
     & + b_h(l_z(\bar{y}_h), (l_{p}(\bar{x}_h),\bar{p}_{h}))
       + b_h(l_z(\bar{x}_h), (l_{p}(\bar{y}_h),\bar{q}_{h}))
    \\
     & - (c_0 l_p(\bar{x}_h), l_p(\bar{y}_h))_{\mathcal{T}_h}
       - \lambda^{-1}\del[1]{\alpha l_p(\bar{x}_h)-l_{p_T}(\bar{x}_h), \alpha l_p(\bar{y}_h) - l_{p_T}(\bar{y}_h)}_{\mathcal{T}_h}.    
  \end{align*}
  Furthermore, let $u_h = u_h^f + l_u(\bar{x}_h)$,
  $p_{Th} = p_{Th}^f + l_{p_T}(\bar{x}_h)$,
  $z_h = z_h^f + l_z(\bar{x}_h)$, and $p_h = p_h^f +
  l_p(\bar{x}_h)$. Then
  $(u_h, \bar{u}_h, p_{Th}, \bar{p}_{Th}, z_h, p_h, \bar{p}_h)$ solves
  \cref{eq:biotcomp}.
\end{lemma}

\subsubsection{A Pressure Auxiliary Problem and its Local Solvers}
\label{sss:localsolversaux}

In our analysis we also require the following reaction-diffusion
boundary value problem:
\begin{equation*}
  -\nabla \cdot (\kappa \nabla \tilde{p}) + (c_0 + \alpha^2 \lambda^{-1})\tilde{p} = \tilde{f} \text{ in } \Omega,
  \qquad
  \tilde{p} = 0 \text{ on } \partial \Omega,
\end{equation*}
where $\tilde{f} \in L^2(\Omega)$ is a given source term. The HDG
discretization of this reaction-diffusion problem is given by: Find
$\boldsymbol{p}_h \in \boldsymbol{Q}_h^0$ such that (see
\cite{wells2011analysis})
\begin{equation}
  \label{eq:HDGreacdif}
  \tilde{a}_h(\boldsymbol{p}_h, \boldsymbol{q}_h) = (\tilde{f}, q_h)_{\mathcal{T}_h} \quad \forall \boldsymbol{q}_h \in \boldsymbol{Q}_h^0,
\end{equation}
where
\begin{multline*}
  \tilde{a}_h(\boldsymbol{p}_h, \boldsymbol{q}_h)
  := \kappa(\nabla p_h, \nabla q_h)_{\mathcal{T}_h}
  - \kappa \langle \nabla p_h \cdot n, q_h - \bar{q}_h \rangle_{\partial\mathcal{T}_h}
  - \kappa \langle \nabla q_h \cdot n, p_h - \bar{p}_h \rangle_{\partial\mathcal{T}_h}
  \\
  + \kappa \eta \langle h_K^{-1}(p_h - \bar{p}_h), q_h - \bar{q}_h \rangle_{\partial\mathcal{T}_h}
  + (c_0 + \alpha^2 \lambda^{-1})(p_h, q_h)_{\mathcal{T}_h}.
\end{multline*}
From \cite[Lemmas 5.2 and 5.3]{wells2011analysis} we know there exist
positive uniform constants $\tilde{c}_1, \tilde{c}_2$ such that
\begin{equation}
  \label{eq:cobndatilde}
  \tilde{c}_1 \tnorm{\boldsymbol{q}_h}_q^2
  \le \tilde{a}_h(\boldsymbol{q}_h, \boldsymbol{q}_h)
  \le \tilde{c}_2 \tnorm{\boldsymbol{q}_h}_q^2
  \quad \forall \boldsymbol{q}_h \in \boldsymbol{Q}_h^0.
\end{equation}

As we did in \cref{sss:localsolversBiot} for the HDG discretization of
the Biot equations, we now define the HDG local solvers for the
reaction-diffusion problem. For this we directly follow
\cite[Definition 2]{henriquez2025parameter}. Given
$\bar{r}_h \in \bar{Q}_h^0$ and $\tilde{s} \in L^2(\Omega)$, we define
the function $\tilde{p}_h^L(\bar{r}_h,\tilde{s}) \in Q_h$ to be such
that its restriction to a cell $K \in \mathcal{T}_h$ satisfies the
local problem
\begin{equation}
  \label{eq:locprobreacdif}
  \tilde{a}_h^L(\tilde{p}_h^L,q_h) = (\tilde{s},q_h)_{\mathcal{T}_h} - \tilde{a}_h((0,\bar{r}_h),(q_h,0))|_K
  \qquad \forall q_h \in Q(K),
\end{equation}
where $\tilde{a}_h^L(p_h,q_h) := \tilde{a}_h((p_h,0),(q_h,0))|_K$. We
now have the following result to eliminate $p_h$ from
\cref{eq:HDGreacdif} (see \cite[Lemma 3]{henriquez2025parameter}).

\begin{lemma}
  \label{lem:tildel}
  Define $\tilde{p}_h^f := \tilde{p}_h^L(0,\tilde{f})$ for given
  $\tilde{f} \in L^2(\Omega)$. For all $\bar{q}_h \in \bar{Q}_h^0$
  define $\tilde{l}_p(\bar{q}_h) := \tilde{p}_h^L(\bar{q}_h,0)$. Let
  $\bar{p}_h \in \bar{Q}_h^0$ solve
  \begin{equation*}
    \tilde{a}_h((\tilde{l}_p(\bar{p}_h),\bar{p}_h),(\tilde{l}_p(\bar{q}_h),\bar{q}_h))
    = (\tilde{f},\tilde{l}_p(\bar{q}_h))_{\mathcal{T}_h}
    \qquad \forall \bar{q}_h \in \bar{Q}_h^0.
  \end{equation*}
  Let $p_h = \tilde{p}_h^f + \tilde{l}_p(\bar{p}_h)$. Then
  $(p_h,\bar{p}_h)$ solves \cref{eq:HDGreacdif}.
\end{lemma}

\subsubsection{A Displacement Auxiliary Problem and its Local Solvers}
\label{sss:localsolversauxdisp}

We now consider the vector diffusion boundary value problem
\begin{equation*}
  - \nabla \cdot \mu \varepsilon(u) = f \text{ in } \Omega,
  \qquad
  u = 0 \text{ on } \partial \Omega,
\end{equation*}
where $f \in L^2(\Omega)^d$ is a given source term. The HDG
discretization of this vector diffusion problem is given by: Find
$\boldsymbol{u}_h \in \boldsymbol{V}_h$ such that
\begin{equation}
  \label{eq:HDGvecdif}
  d_h(\boldsymbol{u}_h, \boldsymbol{v}_h) = (f, q_h)_{\mathcal{T}_h} \quad \forall \boldsymbol{v}_h \in \boldsymbol{V}_h.
\end{equation}

Identical to the pressure auxiliary problem in
\cref{sss:localsolversaux}, given $\bar{r}_h \in \bar{V}_h$ and
$s \in L^2(\Omega)^d$, we define the function
$\tilde{u}_h^L(\bar{r}_h,s) \in V_h$ to be such that its restriction
to a cell $K \in \mathcal{T}_h$ satisfies the local problem
\begin{equation}
  \label{eq:locprobvecdif}
  d_h^L(\tilde{u}_h^L,v_h) = (s,v_h)_{\mathcal{T}_h} - d_h((0,\bar{r}_h),(v_h,0))|_K
  \qquad \forall v_h \in V(K),
\end{equation}
where $d_h^L(u_h,v_h) := d_h((u_h,0),(v_h,0))|_K$. Eliminating $u_h$
from \cref{eq:HDGvecdif} results in the following lemma.

\begin{lemma}
  \label{lem:tildelv}
  Define $\tilde{u}_h^f := \tilde{u}_h^L(0,f)$ for given
  $f \in L^2(\Omega)^d$. For all $\bar{v}_h \in \bar{V}_h$
  define $\tilde{l}_u(\bar{v}_h) := \tilde{u}_h^L(\bar{v}_h,0)$. Let
  $\bar{u}_h \in \bar{V}_h$ solve
  \begin{equation*}
    d_h((\tilde{l}_u(\bar{u}_h),\bar{u}_h),(\tilde{l}_u(\bar{v}_h),\bar{v}_h))
    = (f,\tilde{l}_u(\bar{v}_h))_{\mathcal{T}_h}
    \qquad \forall \bar{v}_h \in \bar{V}_h.
  \end{equation*}
  Let $u_h = \tilde{u}_h^f + \tilde{l}_u(\bar{u}_h)$. Then
  $(u_h,\bar{u}_h)$ solves \cref{eq:HDGvecdif}.
\end{lemma}

\subsection{Various Operators}
\label{ss:ops}

We require various operators. Following the notation in
\cite{henriquez2025parameter}, we first define the operators
$D : \boldsymbol{V}_h \to \boldsymbol{V}_h^{\ast}$ and
$\tilde{A} : \boldsymbol{Q}_h^0 \to \boldsymbol{Q}_h^{0,\ast}$ to be
such that
\begin{align*}
  \langle D \boldsymbol{u}_h, \boldsymbol{v}_h \rangle_{\boldsymbol{V}_h^{\ast},\boldsymbol{V}_h}
  &= d_h(\boldsymbol{u}_h, \boldsymbol{v}_h)
  && \forall \boldsymbol{u}_h,\boldsymbol{v}_h \in \boldsymbol{V}_h,
  \\
  \langle \tilde{A} \boldsymbol{p}_h, \boldsymbol{q}_h \rangle_{\boldsymbol{Q}_h^{0,\ast},\boldsymbol{Q}_h^0}
  &= \tilde{a}_h(\boldsymbol{p}_h, \boldsymbol{q}_h)
  && \forall \boldsymbol{p}_h,\boldsymbol{q}_h \in \boldsymbol{Q}_h^0.
\end{align*}
Using that $\boldsymbol{V}_h := V_h \times \bar{V}_h$ and
$\boldsymbol{Q}_h^0 := Q_h \times \bar{Q}_h^0$, we may write $D$ and
$\tilde{A}$ in block form:
\begin{align*}
  D &=
  \begin{bmatrix}
    D_{11} & D_{21}^T \\
    D_{21} & D_{22}
  \end{bmatrix},
  &&
  \tilde{A} =
  \begin{bmatrix}
    \tilde{A}_{11} & \tilde{A}_{21}^T \\
    \tilde{A}_{21} & \tilde{A}_{22}
  \end{bmatrix},  
\end{align*}
where $D_{11} : V_h \to V_h^{\ast}$,
$D_{21}: V_h \to \bar{V}_h^{\ast}$, and
$D_{22} : \bar{V}_h \to \bar{V}_h^{\ast}$, and where
$\tilde{A}_{11} : Q_h \to Q_h^{\ast}$,
$\tilde{A}_{21}: Q_h \to \bar{Q}_h^{0,\ast}$, and
$\tilde{A}_{22} : \bar{Q}_h^0 \to \bar{Q}_h^{0,\ast}$.  The Schur
complements $S_D : \bar{V}_h \to \bar{V}_h^{\ast}$ and
$S_{\tilde{A}} : \bar{Q}_h^0 \to \bar{Q}_h^{0,\ast}$ of $D$ and
$\tilde{A}$ are defined as $S_D = D_{22} - D_{21}D_{11}^{-1}D_{21}^T$
and
$S_{\tilde{A}} = \tilde{A}_{22} -
\tilde{A}_{21}\tilde{A}_{11}^{-1}\tilde{A}_{21}^T$, respectively. From
\cite[eqns (29) and (44)]{henriquez2025parameter} we have the
following inequalities:
\begin{subequations}
  \begin{align}
    \label{eq:SAtineq}
    \langle \tilde{A}\boldsymbol{q}_h, \boldsymbol{q}_h \rangle_{\boldsymbol{Q}_h^{0,\ast},\boldsymbol{Q}_h^0}
    &\ge \langle S_{\tilde{A}} \bar{q}_h, \bar{q}_h \rangle_{\bar{Q}_h^{0,\ast},\bar{Q}_h^0}
    && \forall \boldsymbol{q}_h \in \boldsymbol{Q}_h^0,
    \\
    \label{eq:SDineq}
    \langle D\boldsymbol{v}_h, \boldsymbol{v}_h \rangle_{\boldsymbol{V}_h^{\ast},\boldsymbol{V}_h}
    &\ge \langle S_D \bar{v}_h, \bar{v}_h \rangle_{\bar{V}_h^{\ast},\bar{V}_h}
    && \forall \boldsymbol{v}_h \in \boldsymbol{V}_h.
  \end{align}
\end{subequations}

Next, we define the operators
$P^u:\boldsymbol{V}_h \to \boldsymbol{V}_h^{\ast}$,
$P^{p_T} : \boldsymbol{Q}_h \to \boldsymbol{Q}_h^{\ast}$,
$P^z : V_h \to V_h^{\ast}$, and
$P^{p} : \boldsymbol{Q}_h^0 \to \boldsymbol{Q}_h^{0,\ast}$ such that
\begin{align*}
  \langle P\boldsymbol{x}_h, \boldsymbol{y}_h \rangle_{\boldsymbol{X}_h^*,\boldsymbol{X}_h}
  =& (\boldsymbol{u}_h,\boldsymbol{v}_h)_{v} + (\boldsymbol{p}_{Th},\boldsymbol{q}_{Th})_{q_T}
     + (z_h,w_h)_{w} + (\boldsymbol{p}_h, \boldsymbol{q}_h)_{q}
  \\
  =& \langle P^u\boldsymbol{u}_h,\boldsymbol{v}_h \rangle_{\boldsymbol{V}_h^{\ast},\boldsymbol{V}_h}
     + \langle P^{p_T}\boldsymbol{p}_{Th},\boldsymbol{q}_{Th} \rangle_{\boldsymbol{Q}_h^{\ast},\boldsymbol{Q}_h}
     + \langle P^zz_h, w_h \rangle_{V_h^{\ast},V_h}
     + \langle P^p\boldsymbol{p}_h, \boldsymbol{q}_h \rangle_{\boldsymbol{Q}_h^{0,\ast},\boldsymbol{Q}_h^0},
\end{align*}
and note that $P$ can be written in block form as:
\begin{equation*}
  P
  =
  \begin{bmatrix}
    P^u & 0 & 0 & 0
    \\
    0 & P^{p_T} & 0 & 0
    \\
    0 & 0 & P^z & 0
    \\
    0 & 0 & 0 & P^p
  \end{bmatrix}
  =
  \begin{bmatrix}
    P_{11}^u & (P_{21}^u)^T & 0 & 0 & 0 & 0 & 0
    \\
    P_{21}^u & P_{22}^u & 0 & 0 & 0 & 0 & 0
    \\
    0 & 0 & P_{11}^{p_T} & 0 & 0 & 0 & 0
    \\
    0 & 0 & 0 & P_{22}^{p_T} & 0 & 0 & 0
    \\
    0 & 0 & 0 & 0 & P^z & 0 & 0
    \\
    0 & 0 & 0 & 0 & 0 & P_{11}^p & (P_{21}^p)^T
    \\
    0 & 0 & 0 & 0 & 0 & P_{21}^p & P_{22}^p
  \end{bmatrix},
\end{equation*}
where $P_{11}^u : V_h \to V_h^{\ast}$,
$P_{21}^u : V_h \to \bar{V}_h^{\ast}$,
$P_{22}^u : \bar{V}_h \to \bar{V}_h^{\ast}$,
$P_{11}^{p_T} : Q_h \to Q_h^{\ast}$,
$P_{22}^{p_T} : \bar{Q}_h \to \bar{Q}_h^{\ast}$,
$P^z : V_h \to V_h^{\ast}$, $P_{11}^p : Q_h \to Q_h^{\ast}$,
$P_{21}^p : Q_h \to \bar{Q}_h^{0,\ast}$, and
$P_{22}^p : \bar{Q}_h \to \bar{Q}_h^{0,\ast}$. The Schur complement
$S_P : (\bar{V}_h \times \bar{Q}_h \times \bar{Q}_h^0) \to
(\bar{V}_h^{\ast} \times \bar{Q}_h^{\ast} \times \bar{Q}_h^{0\ast})$
is defined as:
\begin{equation*}
  S_P =
  \begin{bmatrix}
    S_{P^u} & 0 & 0
    \\
    0 & S_{P^{p_T}} & 0
    \\
    0 & 0 & S_{P^p}
  \end{bmatrix},
\end{equation*}
where $S_{P^u} := P_{22}^u - P_{21}^u(P_{11}^u)^{-1}(P_{21}^u)^T$,
$S_{P^{p_T}} := P_{22}^{p_T}$, and
$S_{P^p} := P_{22}^p - P_{21}^p(P_{11}^p)^{-1}(P_{21}^p)^T$. We recall
from \cite[eqns (31) and (46)]{henriquez2025parameter} that for all
$\boldsymbol{v}_h \in \boldsymbol{V}_h$ and
$\boldsymbol{q}_h \in \boldsymbol{Q}_h^0$,
\begin{subequations}
  \begin{align}
    \label{eq:Puequality}
    \langle P^u\boldsymbol{v}_h, \boldsymbol{v}_h \rangle_{\boldsymbol{V}_h^{\ast},\boldsymbol{V}_h}
    =& \langle P_{11}^u(v_h + (P_{11}^u)^{-1}(P_{21}^u)^T\bar{v}_h), v_h + (P_{11}^u)^{-1}(P_{21}^u)^T\bar{v}_h \rangle_{V_h^{\ast},V_h}
     + \langle S_{P^u}\bar{v}_h, \bar{v}_h \rangle_{\bar{V}_h^{\ast},\bar{V}_h},
    \\
    \label{eq:Ppequality}
    \langle P^p\boldsymbol{q}_h, \boldsymbol{q}_h \rangle_{\boldsymbol{Q}_h^{0,\ast},\boldsymbol{Q}_h^0}
    =& \langle P_{11}^p(q_h + (P_{11}^p)^{-1}(P_{21}^p)^T\bar{q}_h), q_h + (P_{11}^p)^{-1}(P_{21}^p)^T\bar{q}_h \rangle_{Q_h^{\ast},Q_h}
     + \langle S_{P^p}\bar{q}_h, \bar{q}_h \rangle_{\bar{Q}_h^{0,\ast},\bar{Q}_h^0}.
  \end{align}  
\end{subequations}

\subsection{$S_P$ is a Parameter-Robust Preconditioner}
\label{ss:cpproof}

To prove that $S_P$ is a parameter-robust preconditioner for
\cref{eq:SAbiot} we need to prove the existence of a uniform constant
$c_p$ such that \cref{eq:thm23a} holds. To prove this we require the
following result which is proven in \cref{ap:mainresultforcp}.

\begin{lemma}
  \label{lem:mainresultforcp}
  For any
  $\bar{y}_h := (\bar{v}_h,\bar{q}_{Th},\bar{q}_h) \in \bar{X}_h$ let
  $l_u(\bar{y}_h)$, $l_{p_T}(\bar{y}_h)$, $l_z(\bar{z}_h)$, and
  $l_p(\bar{y}_h)$ be as defined in \cref{lem:redvarform}, and
  $\tilde{l}_p(\bar{q}_h)$ and $\tilde{l}_u(\bar{v}_h)$ be as defined
  in \cref{lem:tildel,lem:tildelv}, respectively. There exists a
  uniform constant $c_x$ such that
  \begin{multline}
    \label{eq:mainresultforcp}
    \tnorm{(l_u(\bar{y}_h),\bar{v}_h,l_{p_T}(\bar{y}_h),\bar{q}_{Th},l_z(\bar{y}_h),l_p(\bar{y}_h),\bar{q}_h)}_{\boldsymbol{X}_h}^2
    \\
    \le c_x \eta \del[1]{
       \eta d_h((\tilde{l}_u(\bar{v}_h),\bar{v}_h), (\tilde{l}_u(\bar{v}_h),\bar{v}_h))  
       + \mu^{-1} \eta^{-1} \norm[0]{h_K^{1/2}\bar{q}_{Th}}_{\partial\mathcal{T}_h}^2
       + \tilde{a}_h((\tilde{l}_p(\bar{q}_h),\bar{q}_h),(\tilde{l}_p(\bar{q}_h),\bar{q}_h)) }.
  \end{multline}
\end{lemma}

We now prove our main result.

\begin{theorem}
  \label{eq:cpexists}
  There exists a uniform constant $c_p > 0$ such that \cref{eq:thm23a}
  holds.
\end{theorem}
\begin{proof}
  The proof is similar to that of \cite[Theorem
  3]{henriquez2025parameter}. We can write
  \cref{eq:mainresultforcp} using operator notation as
  \begin{equation}
    \label{eq:mainresultforcp-operator}
    \tnorm{(-A_{11}^{-1}A_{21}^T\bar{y}_h, \bar{y}_h)}_{\boldsymbol{X}_h}^2
    \le c_x \eta \del[1]{ \eta \langle S_D\bar{v}_h, \bar{v}_h\rangle_{V_h^{\ast},V_h}
      + \mu^{-1} \eta^{-1} \norm[0]{h_K^{1/2}\bar{q}_{Th}}_{\partial \mathcal{T}_h}^2
      + \langle S_{\tilde{A}} \bar{q}_h, \bar{q}_h \rangle_{\bar{Q}_h^{0,\ast}, \bar{Q}_h^0} }.
  \end{equation}
  Consider the first term on the right hand side of
  \cref{eq:mainresultforcp-operator}. By \cref{eq:dhbound,eq:SDineq}
  we find
  \begin{equation*}
    \langle P^u\boldsymbol{v}_h, \boldsymbol{v}_h \rangle_{\boldsymbol{V}_h^{\ast},\boldsymbol{V}_h}
    =
    \tnorm{\boldsymbol{v}_h}_v^2
    \ge
    c_1^{-1}d_h(\boldsymbol{v}_h,\boldsymbol{v}_h)
    =
    c_1^{-1} \langle D\boldsymbol{v}_h, \boldsymbol{v}_h \rangle_{\boldsymbol{V}_h^{\ast},\boldsymbol{V}_h}
    \ge
    c_1^{-1} \langle S_D \bar{v}_h, \bar{v}_h \rangle_{\bar{V}_h^{\ast}, \bar{V}_h}
    \quad \forall \boldsymbol{v}_h \in \boldsymbol{V}_h.
  \end{equation*}
  Therefore, choosing $v_h = -(P_{11}^u)^{-1}(P_{21}^u)^T\bar{v}_h$
  and using \cref{eq:Puequality}, we find that
  \begin{equation}
    \label{eq:SpulowSD}
    \langle S_{P^u}\bar{v}_h,\bar{v}_h \rangle_{\bar{V}_h^{\ast},\bar{V}_h}
    \ge c_1^{-1}
    \langle S_D \bar{v}_h, \bar{v}_h \rangle_{\bar{V}_h^{\ast}, \bar{V}_h}
    \quad
    \forall \bar{v}_h \in \bar{V}_h.
  \end{equation}
  Similarly, for the third term on the right hand side of
  \cref{eq:mainresultforcp-operator}, but now using
  \cref{eq:cobndatilde,eq:SAtineq}, we find that
  \begin{equation}
    \label{eq:SpplowStA}
    \langle S_{P^p}\bar{q}_h, \bar{q}_h \rangle_{\bar{Q}_h^{0,\ast},\bar{Q}_h^0}
    \ge \tilde{c}_2^{-1} \langle S_{\tilde{A}} \bar{q}_h, \bar{q}_h \rangle_{\bar{Q}_h^{0,\ast}, \bar{Q}_h^0}
    \quad
    \forall \bar{q}_h \in \bar{Q}_h^0.
  \end{equation}
  For the second term on the right hand side of
  \cref{eq:mainresultforcp-operator} we note that
  \begin{equation}
    \label{eq:Spqterm}
    \mu^{-1} \eta^{-1} \norm[0]{h_K^{1/2}\bar{q}_{Th}}_{\partial \mathcal{T}_h}^2
    =
    \langle S_{P^{p_T}}\bar{q}_h, \bar{q}_h \rangle_{\bar{Q}_h^{\ast},\bar{Q}_h}
    \quad
    \forall \bar{q}_h \in \bar{Q}_h.
  \end{equation}
  Combining
  \cref{eq:mainresultforcp-operator,eq:SpulowSD,eq:SpplowStA,eq:Spqterm} and noting that 
  \begin{align*}
    \langle S_P\bar{y}_h, \bar{y}_h \rangle_{\bar{X}_h^{\ast},\bar{X}_h}
    =& \langle S_{P^u}\bar{v}_h,\bar{v}_h \rangle_{\bar{V}_h^{\ast},\bar{V}_h} + \langle S_{P^{p_T}}\bar{q}_{Th}, \bar{q}_{Th} \rangle_{\bar{Q}_h^{\ast},\bar{Q}_h} + \langle S_{P^p} \bar{q}_h, \bar{q}_h \rangle_{\bar{Q}_h^{0,\ast},\bar{Q}_h^0}
    \\
    =& (\bar{v}_h,\bar{v}_h)_{\bar{V}_h} + (\bar{q}_{Th},\bar{q}_{Th})_{\bar{Q}_h} + (\bar{q}_h,\bar{q}_h)_{\bar{Q}_h^0}
    \\
    =& \norm[0]{\bar{y}_h}_{\bar{X}_h},
  \end{align*}
  the result follows. 
\end{proof}

\Cref{eq:thm23a} and \cref{eq:cpexists} imply that $S_P$ is a
parameter-robust preconditioner for \cref{eq:SAbiot}. Other
parameter-robust preconditioners $S_{\widehat{P}}$ for
\cref{eq:SAbiot} can be found by replacing $P$ by a norm-equivalent
operator $\widehat{P}$ \cite[Remark 3]{henriquez2025parameter}. An
example is the operator
$\widehat{P}^{-1} : \boldsymbol{X}_h^{\ast} \to \boldsymbol{X}_h$
defined by
$\langle \widehat{P}\boldsymbol{x}_h, \boldsymbol{y}_h
\rangle_{\boldsymbol{X}_h^{\ast},\boldsymbol{X}_h} =
\widehat{a}_h(\boldsymbol{x}_h,\boldsymbol{y}_h)$ where
\begin{equation}
  \label{eq:inprodXh-hat}
  \widehat{a}_h(\boldsymbol{x}_h,\boldsymbol{y}_h)
  := d_h(\boldsymbol{u}_h,\boldsymbol{v}_h) + (\boldsymbol{p}_{Th},\boldsymbol{q}_{Th})_{q_T}
  + (z_h,w_h)_{w} + \tilde{a}_h(\boldsymbol{p}_h, \boldsymbol{q}_h).
\end{equation}
The operators $P$ and $\widehat{P}$ are norm-equivalent due to
\cref{eq:dhbound,eq:stabdh,eq:cobndatilde}.

\section{Numerical examples}
\label{s:numex}

In this section we use preconditioned MINRES (with a relative
preconditioned residual tolerance of $10^{-8}$ for two-dimensional
problems and $10^{-6}$ for three-dimensional problems) to solve the
HDG discretization of the Biot equations after static condensation. We
also apply and study preconditioned MINRES for an EDG-HDG
\cite{rhebergen2020embedded} discretization of the Biot equations. The
EDG-HDG discretization is obtained by replacing $\bar{V}_h$ in
\cref{eq:biotcomp} by $\bar{V}_h \cap C^0(\Gamma_0)$. As
preconditioners we consider $S_P$ and $S_{\widehat{P}}$, i.e., the
reduced preconditioners based on the inner products \cref{eq:inprodXh}
and \cref{eq:inprodXh-hat}, respectively. We consider both exact
preconditioners, in which blocks are inverted using a direct solver,
and inexact preconditioners. For the inexact preconditioners, we use
an auxiliary space preconditioner (ASP) with one forward and one
backward Gauss--Seidel smoothing step
\cite{fu2021uniform,fu2023uniform} in combination with Hypre's
BoomerAMG \cite{yang2002boomeramg} for the pressure block, while a
direct solver is used for the mass matrix corresponding to the total
pressure block. For the block corresponding to the displacement we
either use a balancing domain decomposition with constraints (BDDC)
method \cite{schoberl2013domain} (for the EDG-HDG discretization), or
ASP (for the HDG discretization). Furthermore, we consider
unstructured simplicial meshes generated by NETGEN
\cite{schoberl1997netgen}, unless specified otherwise, and the
(EDG-)HDG methods have been implemented in NGSolve
\cite{schoberl2014c++}. As penalty parameter we choose $\eta = 2dk^2$
and all simulations are performed with $k=2$.

\subsection{Test Case 1: Manufactured Solutions}
\label{ss:tc1}

In this first example we set the source terms and boundary conditions
such that the exact displacement and pressure solutions in 2d are
given by
\begin{equation*}
  u =
  \begin{bmatrix}
    \sin(\pi x)\sin(\pi y)  \\
    \sin(\pi x)\cos(\pi y)    
  \end{bmatrix}
  \qquad \text{and} \qquad
  p = \sin(\pi (x - y)),
\end{equation*}
and in 3d by
\begin{equation*}
  u =
  \begin{bmatrix}
    \sin(\pi x)\sin(\pi y)\sin(\pi z)  \\
    \sin(\pi x)\cos(\pi y)\sin(\pi z)  \\
    \sin(\pi x)\cos(\pi y)\cos(\pi z)    
  \end{bmatrix}
  \qquad \text{and} \qquad
  p = \sin(\pi (x - y - z)).
\end{equation*}
In \cref{tab:pre1vpre2} we demonstrate the robustness of the
preconditioners $S_P$ and $S_{\widehat{P}}$ with respect to the mesh
size $h$. For this we choose as parameters $\mu=1$, $\lambda = 10$,
$c_0 = 0.1$, $\kappa = 10^{-4}$, and $\alpha = 0.1$. From
\cref{tab:pre1vpre2} we observe that both preconditioners are
$h$-robust in two and three dimensions, both in their exact and
inexact forms. We further observe that the preconditioners perform
better for an HDG discretization in two dimensions and for an EDG-HDG
discretization in three dimensions. Finally, we observe that
$S_{\widehat{P}}$ always outperforms the preconditioner $S_P$. For
this reason, for the remaining experiments we will only show results
for the preconditioner $S_{\widehat{P}}$.

\begin{table}[tbp]
  \centering
  \resizebox{\textwidth}{!}{%
    \begin{tabular}{c|c|c|c|c|c||c|c|c|c|c}
      \hline
      \multicolumn{11}{c}{EDG-HDG}\\
      \hline
      & \multicolumn{5}{c||}{2d} & \multicolumn{5}{c}{3d} \\
      \hline
      Cells & 608 & 2382 & 9512 & 37938 & 151526 & 48 & 384 & 3072 & 24576 & 196608\\
      \hline
      $S_P$ & 108 (171) & 113 (169) & 108 (166) & 105 (164) & 107 (160) & 123 (125) & 134 (149) & 138 (158) & 134 (158) & 132 (158) \\
      \hline
      $S_{\widehat{P}}$ & 94 (163) & 98 (162) & 95 (157) & 92 (154) & 90 (150) & 105 (118) & 113 (140) & 113 (143) & 108 (142) & 105 (141) \\
      \hline
      \multicolumn{11}{c}{HDG}\\
      \hline
      & \multicolumn{5}{c||}{2d} & \multicolumn{5}{c}{3d} \\
      \hline
      Cells & 608 & 2382 & 9512 & 37938 & 151526 & 48 & 384 & 3072 & 24576 & 196608\\
      \hline
      $S_P$ & 102 (140) & 103 (139) & 98 (132) & 95 (133) & 94 (136) & 93 (101) & 105 (151) & 111 (170) & 111 (178) & 110 (181)\\
      \hline
      $S_{\widehat{P}}$ & 85 (124) & 85 (127) & 81 (123) & 80 (122) & 80 (125) & 82 (98) & 89 (149) & 93 (171) & 93 (178) & 93  (182) \\
      \hline
    \end{tabular}}
  \caption{$h$-robustness for the manufactured solutions test case of
    \cref{ss:tc1}. We list the number of iterations required for
    MINRES, preconditioned by preconditioners $S_P$ and
    $S_{\widehat{P}}$ in both their exact and inexact forms, to
    converge for different mesh sizes $h$. The results for the inexact
    forms of the preconditioners are shown in parentheses.}
  \label{tab:pre1vpre2}
\end{table}

In \cref{tab:paramet-robust} we demonstrate the robustness of the
preconditioner $S_{\widehat{P}}$ with respect to the model
parameters. For this, we vary $\kappa$, $\alpha$, $c_0$, and
$\lambda$, but fix $\mu = 0.5$. We consider a fixed mesh with 9512
simplices in two dimensions and 3072 simplices in three dimensions. We
indeed observe that the preconditioner is robust with respect to the
model parameters. We also again observe that $S_{\widehat{P}}$ performs
better for an HDG discretization in two dimensions and for an EDG-HDG
discretization in three dimensions.

\begin{table}[tbp]
  \centering
  \begin{tabular}{c|c|c||c|c|c|c|c|c}
    \hline
    \multicolumn{9}{c}{EDG-HDG}\\
    \hline
    \multicolumn{3}{c||}{} & \multicolumn{3}{c|}{2d} & \multicolumn{3}{c}{3d}\\
    \hline
    $\kappa$ & $\alpha$ & $c_0$ &$\lambda=10^0$ & $\lambda=10^4$ & $\lambda=10^8$ & $\lambda=10^0$ & $\lambda=10^4$ & $\lambda=10^8$\\
    \hline
    $\multirow{6}{*}{$10^0$}$ & $\multirow{3}{*}{$10^0$}$ & $10^0$ & 73 (142) & 112 (177) & 112 (172) & 82 (108) & 138 (165) & 136 (159)\\
    \cline{3-9}
                           & & $10^{-4}$ & 73 (142) & 112 (177) & 112 (172) & 82 (109) & 136 (165) & 136 (159)\\
    \cline{3-9}
                           & & $0$ & 73 (142) & 112 (177) & 112 (172) & 82 (109) & 136 (165) & 136 (159) \\
    \cline{2-9}
                           & $\multirow{3}{*}{$10^{-4}$}$  & $10^0$ & 71 (143) & 112 (177) & 112 (172) & 80 (108) & 138 (165) & 136 (159) \\
    \cline{3-9}
                           &  & $10^{-4}$ & 72 (142) & 112 (177) & 112 (172) & 80 (108) & 136 (165) & 136 (159) \\
    \cline{3-9}
                           & & $0$ & 72 (143) & 112 (177) & 112 (172) & 80 (108) & 136 (165) & 136 (159) \\
    \cline{1-9}
    $\multirow{6}{*}{$10^{-4}$}$ & $\multirow{3}{*}{$10^0$}$ & $10^0$ & 78 (132) & 112 (177) & 112 (171) & 89 (95) & 136 (159) & 136 (153)\\
    \cline{3-9}
                           & & $10^{-4}$ & 87 (136) & 112 (177) & 112 (172) & 96 (102) & 138 (165) & 138 (160)\\
    \cline{3-9}
                           & & $0$ & 87 (136) & 112 (177) & 112 (172) & 96 (102) & 138 (165) & 136 (160) \\
    \cline{2-9}
                           & $\multirow{3}{*}{$10^{-4}$}$  & $10^0$ & 72 (130) & 112 (176) & 112 (172) & 80 (93) & 136 (159) & 136 (153) \\
    \cline{3-9}
                           &  & $10^{-4}$ & 72 (129) & 112 (174) & 112 (169) & 79 (95) & 138 (163) & 136 (157) \\
    \cline{3-9}
                           & & $0$ & 72 (130) & 112 (174) & 112 (169) & 79 (95) & 136 (162) & 136 (157) \\
    \cline{1-9}
    $\multirow{6}{*}{$10^{-8}$}$ & $\multirow{3}{*}{$10^0$}$ & $10^0$ & 78 (119) & 112 (180) & 112 (154) & 91 (91) & 136 (136) & 136 (136)\\
    \cline{3-9}
                           & & $10^{-4}$ & 102 (141) & 112 (182) & 112 (174) & 113 (113) & 136 (154) & 136 (157)\\
    \cline{3-9}
                           & & $0$ & 102 (141) & 112 (180) & 112 (174) & 113 (113) & 136 (161) & 138 (163)\\
    \cline{2-9}
                           & $\multirow{3}{*}{$10^{-4}$}$  & $10^0$ & 70 (116) & 112 (154) & 112 (154) & 79 (79) & 136 (136) & 136 (136)\\
    \cline{3-9}
                           &  & $10^{-4}$ & 70 (127) & 112 (174) & 112 (169) & 79 (91) & 136 (157) & 136 (151) \\
    \cline{3-9}
                           & & $0$ & 72 (128) & 112 (174) & 112 (169) & 80 (95) & 136 (163) & 136 (157) \\
    \cline{1-9}
    \hline
    \multicolumn{9}{c}{HDG}\\
    \hline
    \multicolumn{3}{c||}{} & \multicolumn{3}{c|}{2d} & \multicolumn{3}{c}{3d}\\
    \hline
    $\kappa$ & $\alpha$ & $c_0$ &$\lambda=10^0$ & $\lambda=10^4$ & $\lambda=10^8$ & $\lambda=10^0$ & $\lambda=10^4$ & $\lambda=10^8$\\
    \hline
    $\multirow{6}{*}{$10^0$}$ & $\multirow{3}{*}{$10^0$}$ & $10^0$ & 74 (129) & 92 (145) & 92 (145) & 90 (222) & 112 (187) & 112 (181)\\
    \cline{3-9}
                           & & $10^{-4}$ & 74 (129) & 92 (145) & 92 (145) & 90 (222) & 112 (187) & 112 (182)\\
    \cline{3-9}
                           & & $0$ & 74 (129) & 92 (145) & 92 (145) & 90 (222) & 112 (187) & 112 (182) \\
    \cline{2-9}
                           & $\multirow{3}{*}{$10^{-4}$}$  & $10^0$ & 74 (129) & 92 (145) & 92 (145) & 89 (222) & 112 (187) & 112 (182) \\
    \cline{3-9}
                           &  & $10^{-4}$ & 74 (129) & 92 (145) & 92 (145) & 89 (222) & 112 (187) & 112 (182) \\
    \cline{3-9}
                           & & $0$ & 74 (129) & 92 (145) & 92 (145) & 89 (222) & 112 (187) & 112 (182) \\
    \cline{1-9}
    $\multirow{6}{*}{$10^{-4}$}$ & $\multirow{3}{*}{$10^0$}$ & $10^0$ & 77 (124) & 92 (145) & 92 (145) & 90 (161) & 112 (180) & 112 (175) \\
    \cline{3-9}
                           & & $10^{-4}$ & 82 (129) & 92 (145) & 92 (145) & 99 (174) & 112 (180) & 112 (182) \\
    \cline{3-9}
                           & & $0$ & 82 (129) & 92 (145) & 92 (145) & 99 (174) & 112 (187) & 112 (182) \\
    \cline{2-9}
                           & $\multirow{3}{*}{$10^{-4}$}$  & $10^0$ & 67 (113) & 92 (145) & 92  (145) & 80 (160) & 112 (180) & 112 (174) \\
    \cline{3-9}
                           &  & $10^{-4}$ & 65 (109) & 92 (145) & 92 (145) & 73 (164) & 112 (184) & 112 (179) \\
    \cline{3-9}
                           & & $0$ & 65 (109) & 92 (145) & 92 (145) & 73 (164) & 112 (184) & 112 (179) \\
    \cline{1-9}
    $\multirow{6}{*}{$10^{-8}$}$ & $\multirow{3}{*}{$10^0$}$ & $10^0$ & 79 (124) & 92 (145) & 92 (145) & 81 (131) & 112 (162) & 112 (162) \\
    \cline{3-9}
                           & & $10^{-4}$ & 102 (151) & 92 (145) & 92 (145) & 103 (155) & 113 (175) & 112 (177) \\
    \cline{3-9}
                           & & $0$ & 102 (151) & 92 (145) & 92 (145) & 103 (155) & 113 (182) & 112 (184)\\
    \cline{2-9}
                           & $\multirow{3}{*}{$10^{-4}$}$  & $10^0$ & 69 (115) & 92 (145) & 92  (145)& 71 (127) & 112 (162) & 112 (162) \\
    \cline{3-9}
                           &  & $10^{-4}$ & 59 (100) & 92 (145) & 92 (145) & 63 (126) & 112 (177) & 112 (172) \\
    \cline{3-9}
                           & & $0$ & 60 (100) & 92 (145) & 92 (145)  & 71 (151) & 112 (184) & 112 (179) \\
    \cline{1-9}
  \end{tabular}
  \caption{parameter-robustness for the manufactured solutions test
    case of \cref{ss:tc1}. We list the number of iterations required
    for MINRES, preconditioned by $S_{\widehat{P}}$ in both its exact
    and inexact form, to converge for different parameter values. The
    results for the inexact form of $S_{\widehat{P}}$ is shown in
    parentheses.}
  \label{tab:paramet-robust}
\end{table}

\subsection{Test Case 2: The Footing Problem}
\label{ss:tc2}

We now consider the 2d \cite{gaspar2008stabilized,oyarzua2016locking}
and 3d \cite{boon2021robust,gaspar2008distributive} footing
problems. These are benchmark problems to test the locking-free
properties of numerical methods for Biot's consolidation problem. We
consider the time-dependent problem in which we use backward Euler
time-stepping. We fix the model parameters, but vary the time step
$\tau$. We have the following problem setup:
\begin{itemize}
\item[$\bullet$] 2d setup. The domain is given by
  $\Omega = (-50,50) \times (0,75)$ with boundaries
  $\Gamma_1 = \{(x,y) \in \partial \Omega, |x| \leq 50/3, y = 75\}$,
  $\Gamma_2 = \{(x,y) \in \partial \Omega, |x| > 50/3, y = 75\}$, and
  $\Gamma_3 = \partial \Omega \backslash(\Gamma_1 \cup \Gamma_2)$. We
  consider a final time of $T=50$ and compute the solution on a mesh
  consisting of 48345 cells. We choose the following model parameters:
  $\tilde{\kappa} = 10^{-4}$, $c_0 = 10^{-3}$, $\alpha = 0.1$,
  $E = 3 \times 10^4 $, and $\nu = 0.4995$. Here we remark that $E$ is
  Young's modulus of elasticity and $\nu$ is Poisson's ratio. In the
  case of plane strain, these are related to the Lam\'e constants
  through $\lambda = E\nu / ((1+\nu)(1-2\nu))$ and
  $\tilde{\mu} = E/(2(1+\nu))$.
\item[$\bullet$] 3d setup. The domain is given by
  $\Omega = (-32,32) \times (-32,32) \times (0,64)$ with boundaries
  $\Gamma_1 = \{(x,y,z) \in \partial \Omega, |x| \leq 16, |y| \leq 16,
  z = 64\}$,
  $\Gamma_2 = \{(x,y,z) \in \partial \Omega, |x| > 16, |y| > 16, z =
  64\}$ and
  $\Gamma_3 = \partial \Omega \backslash(\Gamma_1 \cup \Gamma_2)$. We
  consider a final time of $T=1$ and compute the solution on a mesh
  consisting of 16520 cells. We choose the following model parameters:
  $\tilde{\kappa} = 10^{-7}$, $c_0 = 0.5$, $\alpha = 0.5$,
  $E = 3 \times 10^4$, and $\nu = 0.45$.
\end{itemize}
Let
$\sigma = \mu \varepsilon(u) + \lambda \nabla \cdot u \mathbb{I} -
\alpha p \mathbb{I}$. The boundary conditions are given by
\begin{equation*}
  \sigma n = (0, -\sigma_0)^T\text{ on }\Gamma_1,\quad
  \sigma n = 0 \text{ on }\Gamma_2, \quad
  u = 0\text{ on }\Gamma_3, \quad
  p = 0\text{ on }\partial \Omega,     
\end{equation*}
with $\sigma_0 = 10^4$ in 2d and $\sigma_0=0.1$ in 3d. As initial
conditions we impose $u(x,0) = 0$ and $p(x,0) = 0$.

In \cref{tab:footing-prob} we list the average number of iterations
per time step required for MINRES to converge with preconditioner
$S_{\widehat{P}}$ for varying $\tau$. Once again, the preconditioner
is observed to be robust.

\begin{table}[tbp]
  \centering
  \begin{tabular}{c|c|c|c|c|c}
    \hline
    $\tau$ & 1 & 0.25 & 0.025 & 0.0025 & 0.0001 \\
    \hline
    EDG-HDG (2d) & 116 (251) & 116 (244) & 116 (225) & 116 (200) & 116 (200) \\
    \hline
    HDG (2d) & 95 (210) & 95 (206) & 95 (187) & 95 (171) & 95 (171) \\
    \hline
    EDG-HDG (3d) & 133 (224) & 133 (223) & 133 (223) & 133 (223) & 133 (223) \\
    \hline
    HDG (3d) & 87 (184) & 87 (184) & 87 (184) & 87 (184) & 87 (184) \\    
    \hline
  \end{tabular}
  \caption{The footing problem of \cref{ss:tc2}. We list the number of
    iterations required for MINRES, preconditioned by
    $S_{\widehat{P}}$ in both its exact and inexact form, to converge
    for different time steps $\tau$. The results for the inexact form
    of the preconditioner is shown in parentheses. }
  \label{tab:footing-prob}
\end{table}

\subsection{Test Case 3: A Simplified Brain Model}
\label{ss:tc3}

In this final example we study the performance of the preconditioner
$S_{\widehat{P}}$ when applied to a simplified 3D simulation of the
cerebrospinal fluid-tissue interaction in a human brain. We compute
the solution on the mesh presented in \cite[Chapter
3]{mardal2022mathematical}, which consists of 155953 cells. In
\cref{eq:biotcomp} we fix the parameters $c_0 = 3 \times 10^{-4}$,
$\alpha = 0.25$, and $E = 1500$, but vary the parameters $\nu$ and
$\kappa$. As boundary conditions we impose $u=0$ and $p=0.001$ on
$\partial \Omega$. When using ASP for the HDG discretization we take
two forward and two backward block Gauss--Seidel smoothing steps; for
EDG-HDG we take one forward and one backward block Gauss--Seidel
smoothing step. In \cref{fig:bothimages} we plot the displacement and
pressure. In \cref{tab:brain} we list the number of iterations
required for MINRES to converge for varying $\nu$ and
$\kappa$. Although we observe some variation in iteration count when
varying $\nu$ and $\kappa$, the variations are small.

\begin{figure}[tbp]
  \centering
  \begin{minipage}{0.45\textwidth}
    \centering
    \includegraphics[width=\linewidth]{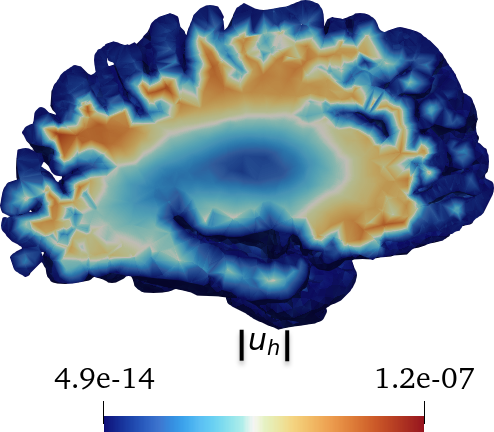}
  \end{minipage}
  \hfill
  \begin{minipage}{0.45\textwidth}
    \centering
    \includegraphics[width=\linewidth]{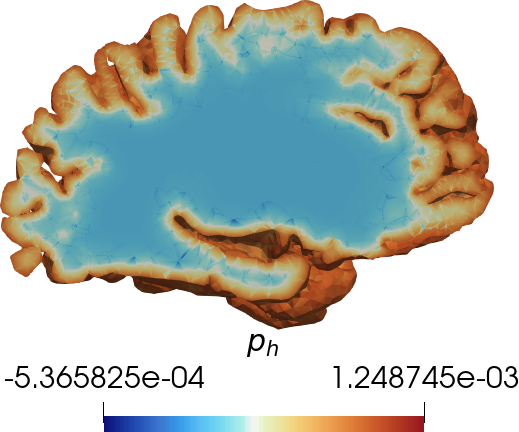}
  \end{minipage}
  \caption{Approximation of the displacement (left) and pressure
    (right) using the EDG-HDG discretization with $\nu = 0.45$ and
    $\kappa = 1.57\times 10^{-1}$.}
  \label{fig:bothimages}
\end{figure}

\begin{table}[tbp]
  \centering
  \begin{tabular}{c|c|c|c|c|c|c}
    \hline
    &\multicolumn{3}{c|}{EDG-HDG}&\multicolumn{3}{c}{HDG}\\
    \hline
    $\nu / \kappa$ & $3.75$ & $1.57\times 10^{-1}$ & $1.57\times 10^{-3}$& $3.75$ & $1.57\times 10^{-1}$ & $1.57\times 10^{-3}$ \\
    \hline
    0.34 & 116 (147) & 117 (146) & 148 (178) & 165 $(382)$ & 165 $(392)$ & 208 $(492)$ \\
    \hline
    0.45 & 156 (197) & 157 (196) & 200 (239) & 229 $(467) $& 231 $(472)$ & 290 $(590)$ \\
    \hline
  \end{tabular}
  \caption{The simplified-poroelastic brain model, see
    \cref{ss:tc3}. We list the number of iterations required for
    MINRES, preconditioned by $S_{\widehat{P}}$ in both its exact and
    inexact form, to converge for varying $\nu$ and $\kappa$. The
    results for the inexact form of the preconditioner is shown in
    parentheses.}
  \label{tab:brain}
\end{table}

\section{Conclusions}
\label{s:conclusions}

In this paper, we introduced parameter-robust preconditioners for the
reduced linear system derived from an HDG discretization of a
four-field formulation of the Biot-consolidation model. We first
developed a preconditioner for the non-condensed linear system of the
discretization using the Mardal and Winther framework
\cite{mardal2011preconditioning}. Next, we used the extension of this
framework to hybridized methods \cite{henriquez2025parameter} to
obtain a parameter-robust preconditioner for the reduced linear
system. Numerical examples in two and three dimensions confirm
robustness of the preconditioners for the condensed system.

\section*{Funding}

Jeonghun J. Lee acknowledges support from the National Science
Foundation through grant number DMS-2110781 and Sander Rhebergen
acknowledges support from the Natural Sciences and Engineering
Research Council of Canada through the Discovery Grant program
(RGPIN-2023-03237).

\bibliographystyle{plain}
\bibliography{refs}
\appendix
\section{Proof of \cref{lem:mainresultforcp}}
\label{ap:mainresultforcp}

Let us define the following norms on $\bar{V}_h$ and $\bar{Q}_h$,
respectively:
\begin{align*}
  \tnorm{\bar{v}_h}_{h,u}^2
  &:= \sum_{K \in \mathcal{T}_h} h_K^{-1}\norm[0]{\bar{v}_h - m_K(\bar{v}_h)}_{\partial K}^2
  && \forall \bar{v}_h \in \bar{V}_h,
  \\
  \tnorm{\bar{q}_h}_{h,p}^2
  &:= \sum_{K \in \mathcal{T}_h} h_K^{-1}\norm[0]{\bar{q}_h - m_K(\bar{q}_h)}_{\partial K}^2
  && \forall \bar{q}_h \in \bar{Q}_h,
\end{align*}
where $m_K(w) := |\partial K|^{-1}\langle w, 1 \rangle_{\partial
  K}$. Furthermore, we define $m(w)$ such that $m(w)|_K = m_K(w)$ for
$K \in \mathcal{T}_h$. We then recall the following results.

\begin{lemma}
  \label{lem:estimate-m_K}
  Let $\bar{q}_h \in \bar{Q}_h$ be given and let
  $\tilde{l}_p(\bar{q}_h)$ be as defined in
  \cref{lem:tildel}. There exists a uniform constant
  $c > 0$ such that
  \begin{align}
    \label{eq:estimate-m_K}
    \min\cbr[0]{c_0 + \alpha^2 \lambda^{-1}, h_K^{-2}\kappa} \norm[0]{m_K(\bar{q}_h)}_{ K}^2
    & \le c \del[1]{ (c_0 + \alpha^2 \lambda^{-1}) \norm[0]{\tilde{l}_p(\bar{q}_h)}_K^2
      + \kappa \eta h^{-1}_K \norm[0]{\tilde{l}_p(\bar{q}_h) - \bar{q}_h}_{\partial K}^2},
    \\
    \label{eq:tildeah-estimate}
    \kappa \tnorm{\bar{q}_h}_{h,p}^2 & \le c \tilde{a}_h((\tilde{l}_p(\bar{q}_h),\bar{q}_h),(\tilde{l}_p(\bar{q}_h),\bar{q}_h)).
  \end{align}
\end{lemma}
\begin{proof}
  Estimate \cref{eq:estimate-m_K} is proven in
  \cite[Lemma 9]{henriquez2025parameter}. Estimate
  \cref{eq:tildeah-estimate} is obtained by taking the same arguments
  as that of \cite[Lemma~5]{rhebergen2018preconditioning} and so we
  omit the details.
\end{proof}

We now prove \cref{lem:mainresultforcp}.

\begin{proof}[of \cref{lem:mainresultforcp}]
  We split the proof into three steps. 
  \\
  \textbf{Step 1.} By definition of
  $\tnorm{\cdot}_{\boldsymbol{X}_h}$, we claim that it is sufficient
  to prove existence of a uniform $c>0$ such that
  \begin{align}
    \notag
    &\tnorm{(l_u(\bar{y}_h),\bar{v}_h,l_{p_T}(\bar{y}_h),\bar{q}_{Th},l_z(\bar{y}_h),l_p(\bar{y}_h),\bar{q}_h)}_{\boldsymbol{X}_h}^2
    \\
    \notag
    =
    & \tnorm{(l_u(\bar{y}_h),\bar{v}_h)}_v^2
      + \tnorm{(l_{p_T}(\bar{y}_h),\bar{q}_{Th})}_{q_T}^2
      + \tnorm{l_z(\bar{y}_h)}_w^2
      + \tnorm{(l_p(\bar{y}_h),\bar{q}_h)}_q^2
    \\
    \label{eq:desired-estimate}
    \le& c \eta \del[2]{
       \eta \mu \tnorm{\bar{v}_h}_{h,u}^2
       + \mu^{-1} \eta^{-1} \norm[0]{h_K^{1/2}\bar{q}_{Th}}_{\partial\mathcal{T}_h}^2
       + \tilde{a}_h((\tilde{l}_p(\bar{q}_h),\bar{q}_h),(\tilde{l}_p(\bar{q}_h),\bar{q}_h)) }.
  \end{align}
  The desired result then follows by noting that there exists a
  uniform constant $C>0$ such that
  $C \mu \tnorm{\bar{v}_h}_{h,u}^2 \le
  d_h((\tilde{l}_u(\bar{v}_h),\bar{v}_h),
  (\tilde{l}_u(\bar{v}_h),\bar{v}_h))$ for all
  $\bar{v}_h \in \bar{V}_h$ (see \cite[Lemma
  5]{rhebergen2018preconditioning}).
  \\
  From the bilinear form \cref{eq:definitionah} in the local problem
  \cref{eq:localproblembiot},
  \begin{subequations}
    \label{eq:local-ah-reform-eqs}
    \begin{align}
      \label{eq:local-ah-reform-eq1}
      (\mu \varepsilon(l_{u}(\bar{y}_h)), \varepsilon({v}'))_{K}
      +\mu \eta h_K^{-1}\langle (l_{u}(\bar{y}_h) - \bar{v}_h),{v}'\rangle_{\partial K}
      \\
      \notag
      - \langle \mu \varepsilon(l_{u}(\bar{y}_h))n, {v}'\rangle_{\partial K}
      - \langle \mu \varepsilon({v}')n, l_{u}(\bar{y}_h)-\bar{v}_h\rangle_{\partial K}   
      - (l_{p_T}(\bar{y}_h), \nabla \cdot v')_K
      & = - \langle \bar{q}_{Th}, v' \cdot n\rangle_{\partial K}, 
      \\
      \label{eq:local-ah-reform-eq2}
      - (q_T', \nabla \cdot l_u(\bar{y}_h))_K
      + \lambda^{-1}(\alpha l_p(\bar{y}_h)
      - l_{p_T}(\bar{y}_h), q_T')_K
      & = 0, 
      \\
      \label{eq:local-ah-reform-eq3}
      \del[1]{\kappa^{-1}l_z(\bar{y}_h), w'}_{K}     
      - (l_p(\bar{y}_h), \nabla \cdot w')_K
      & = -\langle \bar{q}_h, w'\cdot n\rangle_{\partial K},
      \\
      \label{eq:local-ah-reform-eq4}
      - (c_0 l_p(\bar{y}_h), q')_{K}
      - (q', \nabla \cdot l_z(\bar{y}_h))_K
      - \lambda^{-1}\del[1]{\alpha l_{p}(\bar{y}_h)
      - l_{p_T}(\bar{y}_h), \alpha q' }_{K}
      & = 0,
    \end{align}
  \end{subequations}
  for all
  $(v', q_T', w', q')\in V(K)\times Q(K)\times V(K) \times Q(K)$. For
  ease of notation we
  introduce
  \begin{equation*}
    \Psi^0|_K := l_{\Psi}(0,0,m_K(\bar{q}_h)),
    \qquad
    \Psi^1|_K := l_{\Psi}(\bar{v}_h,\bar{q}_{Th},\bar{q}_h - m_K(\bar{q}_h)),
    \qquad \text{for }\Psi \in \cbr[0]{u,p_T,z,p}
  \end{equation*}
  and define $m(\bar{q}_h)$ by the function from $\mathcal{T}_h$ to
  $\mathbb{R}$ such that its value on $K$ is $m_K(\bar{q}_h)$. In the
  rest of proof we will show that
  $\tnorm{(u^0, \bar{0},p_T^0, \bar{0}, z^0, p^0,
    m(\bar{q}_h))}_{\boldsymbol{X}_h}^2 \le c \eta
  \tilde{a}_h((\tilde{l}_p(\bar{q}_h),\bar{q}_h),(\tilde{l}_p(\bar{q}_h),\bar{q}_h))$
  and
  \begin{multline}
    \label{eq:A3formu1}
    \tnorm{(u^1, \bar{v}_h,p_T^1, \bar{q}_{Th}, z^1, p^1, \bar{q}_h
      -m(\bar{q}_h))}_{\boldsymbol{X}_h}^2
    \\
    \le c \eta \del[2]{
      \eta \mu \tnorm{\bar{v}_h}_{h,u}^2
      + \mu^{-1} \eta^{-1} \norm[0]{h_K^{1/2}\bar{q}_{Th}}_{\partial\mathcal{T}_h}^2
      + \tilde{a}_h((\tilde{l}_p(\bar{q}_h),\bar{q}_h),(\tilde{l}_p(\bar{q}_h),\bar{q}_h)) }.
  \end{multline}
  \Cref{eq:desired-estimate} then follows by the triangle inequality.
  \\
  \textbf{Step 2.} To prove
  $\tnorm{(u^0, \bar{0},p_T^0, \bar{0}, z^0, p^0,
    m(\bar{q}_h))}_{\boldsymbol{X}_h}^2 \le
  c \eta \tilde{a}_h((\tilde{l}_p(\bar{q}_h),\bar{q}_h),(\tilde{l}_p(\bar{q}_h),\bar{q}_h))$
  we need to prove
  \begin{equation}
    \label{eq:desired-estimate-0}
    \tnorm{(u^0, \bar{0},p_T^0, \bar{0}, z^0, p^0, m(\bar{q}_h))}_{\boldsymbol{X}_h}^2
    \le c \eta \sum_{K\in \mathcal{T}_h} \min \cbr{c_0 + \alpha^2 \lambda^{-1}, h_K^{-2} \kappa }
    \norm[0]{m_K(\bar{q}_h)}_K^2.
  \end{equation}
  The result then follows by \cref{eq:estimate-m_K,eq:cobndatilde}. We
  therefore proceed by proving \cref{eq:desired-estimate-0}.
  
  By the definitions of the $\Psi^0$'s and
  \cref{eq:local-ah-reform-eqs} with $v'=u^0$, $q_T' = -{p_T}^0$,
  $w'=z^0$, $q'=-p^0$, we find that
  \begin{multline}
    \label{eq:u0p0pT0z0}
    (\mu \varepsilon({u}^0), \varepsilon(u^0))_{K}
    +\mu \eta h_K^{-1}\langle
    {u}^0,u^0\rangle_{\partial K} - 2\langle \mu \varepsilon({u}^0)n,
    u^0\rangle_{\partial K} + (\kappa^{-1}z^0, z^0)_K + (c_0 p^0, p^0)_K
    \\
    + \lambda^{-1}(\alpha p^0 - {p_T}^0, \alpha p^0 - {p_T}^0)_K
    = -\langle m_K(\bar{q}_h), z^0 \cdot n\rangle_{\partial K}.
  \end{multline}
  By the Cauchy--Schwarz inequality, a discrete trace inequality (see
  \cite[Lemma 1.46]{di2011mathematical}) and Young's inequality, we
  find
  \begin{equation}
    \label{eq:Psi0-intermediate-1}
    |-\langle m_K(\bar{q}_h), z^0 \cdot n\rangle_{\partial K}|
    \le
    C \kappa h_K^{-2} \norm[0]{m_K(\bar{q}_h)}_K^2
    + \tfrac{1}{2} (\kappa^{-1}z^0, z^0)_K.
  \end{equation}
  Furthermore,
  $\langle m_K(\bar{q}_h), z^0 \cdot n\rangle_{\partial
    K}=(m_K(\bar{q}_h), \nabla \cdot z^0 )_K$ and
  $\nabla \cdot z^0 = -c_0 p^0 - \lambda^{-1}\alpha (\alpha p^0 -
  p_T^0)$. Combining these two identities and using Young's inequality we find that
  \begin{equation}
    \label{eq:Psi0-intermediate-2}
    |-\langle m_K(\bar{q}_h), z^0 \cdot n\rangle_{\partial K}| 
    \le C (c_0 + \lambda^{-1}\alpha^2) \norm[0]{m_K(\bar{q}_h)}_K^2
    + \tfrac{1}{2} (c_0p^0, p^0)_K
    + \tfrac{1}{2} \lambda^{-1}(\alpha p^0 - p_T^0, \alpha p^0 - p_T^0)_K.
  \end{equation}
  Combining now
  \cref{eq:u0p0pT0z0,eq:Psi0-intermediate-1,eq:Psi0-intermediate-2},
  \begin{multline}
    \label{eq:u0p0pT0z0ineq}
    (\mu \varepsilon({u}^0), \varepsilon(u^0))_{K}
    +\mu \eta h_K^{-1}\langle {u}^0,u^0\rangle_{\partial K}
    - 2\langle \mu \varepsilon({u}^0)n, u^0\rangle_{\partial K}
    + (\kappa^{-1}z^0, z^0)_K
    + (c_0 p^0, p^0)_K
    \\
    + \lambda^{-1}(\alpha p^0 - {p_T}^0, \alpha p^0 - {p_T}^0)_K
    \le c \min \cbr[1]{c_0 + \alpha^2\lambda^{-1},  h_K^{-2}\kappa } \norm[0]{m_K(\bar{q}_h)}_K^2.
  \end{multline}
  By \cref{eq:stabdh} and \cref{eq:u0p0pT0z0ineq} with summation over
  $K\in \mathcal{T}_h$,
  \begin{equation}
    \label{eq:result1}
    \tnorm{(u^0, \bar{0})}_v^2
    + \tnorm{z^0}_w^2
    + c_0\norm[0]{p^0}_{\mathcal{T}_h}^2 + \lambda^{-1}\norm[0]{\alpha p^0 - p_T^0}_{\mathcal{T}_h}^2
    \le c \sum_{K\in \mathcal{T}_h} \min \cbr[1]{c_0 + \alpha^2 \lambda^{-1}, h_K^{-2} \kappa} \norm[0]{m_K(\bar{q}_h)}_K^2.
  \end{equation}
  Next, we prove the following two inequalities:
  \begin{subequations}
    \label{eq:proofcomplete-ab}
    \begin{align}
      \label{eq:proofcomplete-a}
      \kappa^{1/2} \tnorm{(p^0, m(\bar{q}_h))}_{1,p} &\le c \eta^{1/2} \tnorm{z^0}_w,
      \\
      \label{eq:proofcomplete-b}
      \tnorm{(p_T^0, \bar{0})}_{q_T} &\le c\tnorm{(u^0, \bar{0})}_v.
    \end{align}
  \end{subequations}
  To prove \cref{eq:proofcomplete-a}, we follow steps similar to those
  of the proof of \cite[Lemma 4.2]{cesmelioglu2024strongly}. By
  definition of the local degrees of freedom of BDM elements
  \cite[Prop. 2.3.2]{boffi2013mixed}, given
  $(p^0,\bar{q}_h) \in \boldsymbol{Q}_h^0$, we define the BDM function
  $\tilde{z}_h \in V_h \cap H(\text{div};\Omega)$ such that
  \begin{subequations}
    \label{eq:deftildezh}
    \begin{align}
      (\tilde{z}_h, w_h)_K &= \kappa (\nabla p^0, w_h)_K
      && \forall w_h \in \mathcal{N}_{k-2}(K), \ \forall K \in \mathcal{T}_h,
      \\
      \langle \tilde{z}_h\cdot n, \bar{w}_h\rangle_{\partial K}
                           &= \kappa \eta h_K^{-1}\langle m_K(\bar{q}_h) - p^0, \bar{w}_h\ \rangle_{\partial K}
      && \forall \bar{w}_h \in \mathcal{R}_K(\partial K), \ \forall K \in \mathcal{T}_h,
    \end{align}
  \end{subequations}
  where $\mathcal{N}_{k-2}(K)$ is the N\'ed\'elec space and
  $\mathcal{R}_k(\partial K) := \cbr[0]{\bar{w} \in L^2(\partial
    K)\, : \, \bar{w}|_F \in P_k(F),\ \forall F \subset \partial
    K}$. Choosing $w_h = \nabla p^0$ and
  $\bar{w}_h = m_K(\bar{q}_h) - p^0$ in \cref{eq:deftildezh} we
  obtain
  \begin{subequations}
    \begin{align}
      \label{eq:tildezhgradp0-a}
      (\tilde{z}_h, \nabla p^0)_K &= \kappa \norm[0]{\nabla p^0}_K^2 && \forall K \in \mathcal{T}_h,
      \\
      \label{eq:tildezhgradp0-b}
      \langle \tilde{z}_h\cdot n, m_K(\bar{q}_h) - p^0\rangle_{\partial K}
                                  &= \kappa \eta h_K^{-1} \norm[0]{m_K(\bar{q}_h)-p^0}_{\partial K}^2
                                                                     && \forall K \in \mathcal{T}_h.
    \end{align}            
  \end{subequations}
  Adding \cref{eq:tildezhgradp0-a,eq:tildezhgradp0-b}, integrating by
  parts, and summing over all $K \in \mathcal{T}_h$, we find
  \begin{equation}
    \label{eq:p0mkbarq-a1n}
    b_h(\tilde{z}_h, (p^0, m(\bar{q}_h)))
    = \kappa \tnorm{ (p^0, m(\bar{q}_h)) }_{1,p}^2.
  \end{equation}
  Next, similar to the proof of \cite[Lemma 4.4]{linke2018quasi}, we
  have
  \begin{align*}
    \norm[0]{\tilde{z}_h}_K^2
    + \sum_{F \in \mathcal{F}_K}h_F\norm[0]{\tilde{z}_h \cdot n}_F^2
    &\lesssim \sup_{\substack{ w_h \in \mathcal{N}_{k-2}(K)^3 \\ \norm[0]{w_h}_{K}=1}} | (\tilde{z}_h, w_h)_K |^2
    + \sup_{\substack{ \bar{w}_h \in R_k(\partial K) \\ \norm[0]{\bar{w}_h}_{\partial K}=1}} h_K |\langle \tilde{z}_h \cdot n, \bar{w}_h \rangle_F |^2
    \\
    &= \sup_{\substack{ w_h \in \mathcal{N}_{k-2}(K)^3 \\ \norm[0]{w_h}_{K}=1}} \kappa^2 | (\nabla p^0, w_h)_K |^2
    +\sup_{\substack{ \bar{w}_h \in R_k(\partial K) \\ \norm[0]{\bar{w}_h}_{\partial K}=1}} \kappa^2 \eta^2 h_Kh_K^{-2} |\langle m_K(\bar{q}_h)-p^0, \bar{w}_h \rangle_{\partial K} |^2
    \\
    &\le \kappa^2\norm[0]{\nabla p^0}_{K}^2+ \kappa^2\eta^2h_K^{-1} \norm[0]{m_K(\bar{q}_h) - p^0}_{\partial K}^2
  \end{align*}
  Dividing by $\kappa$ and summing over all $K \in \mathcal{T}_h$, we
  find that
  \begin{equation}
    \label{eq:p0mkbarq-bnn}
    \tnorm{\tilde{z}_h}_w
    \le c \kappa^{1/2} \eta^{1/2} \tnorm{(p^0,m(\bar{q}_h))}_{1,p}.
  \end{equation}
  Then, by definition of $z^0,p^0$, we find from
  \cref{eq:local-ah-reform-eq3}, after summing over all cells
  $K \in \mathcal{T}_h$, that
  $b_h(w', (p^0,m(\bar{q}_h))) =
  \kappa^{-1}(z^0,w')_{\mathcal{T}_h}$. Choosing $w'=\tilde{z}_h$,
  using \cref{eq:p0mkbarq-a1n}, the Cauchy--Schwarz inequality, and
  \cref{eq:p0mkbarq-bnn}, we obtain \cref{eq:proofcomplete-a}.
  
  To prove \cref{eq:proofcomplete-b} we note that a consequence of
  \cref{eq:stabbhQh} is that for given
  $(p_T^0,\bar{0}) \in \boldsymbol{Q}_h$ there exists
  $\tilde{\boldsymbol{v}}_h = (\tilde{v}_h, \tilde{\bar{v}}_h) \in
  \boldsymbol{V}_h$ such that
  \begin{subequations}
    \begin{align}
      \label{eq:pT0vtilde_a}
      \tnorm{(p_T^0, \bar{0})}_{q_T}^2
      & = b_h(\tilde{v}_h, (p_T^0,\bar{0}))
        = -(p_T^0,\nabla \cdot \tilde{v}_h)_{\mathcal{T}_h},
      \\
      \label{eq:pT0vtilde_b}
      \tnorm{\tilde{\boldsymbol{v}}_h}_v
      &\le c \tnorm{(p_T^0, \bar{0})}_{q_T}.
    \end{align}
  \end{subequations}
  Recalling that $(u^0, p_T^0)$ satisfies
  \cref{eq:local-ah-reform-eq1} with vanishing $\bar{v}_h$ and
  $\bar{q}_{Th}$, the summation over $K$ of this equation with
  $v'|_K = \tilde{v}_h - m_K(\tilde{\bar{v}}_h)$, and using
  \cref{eq:pT0vtilde_a} gives
  \begin{equation}
    \label{eq:dhu0bar0}
    d_h((u^0, \bar{0}), (\tilde{v}_h - m(\tilde{\bar{v}}_h), \bar{0}))
    = -\tnorm{(p_T^0, \bar{0})}_{q_T}^2. 
  \end{equation}
  By \cref{eq:dhbound} we have
  \begin{equation}
    \label{eq:dhu0bar0bound}
    |d_h((u^0, \bar{0}), (\tilde{v}_h - m(\tilde{\bar{v}}_h), \bar{0}))|
    \le c \tnorm{(u^0, \bar{0})}_v \tnorm{(\tilde{v}_h - m(\tilde{\bar{v}}_h), \bar{0})}_v.
  \end{equation}
  Note that
  $\tnorm{(\tilde{v}_h - m(\tilde{\bar{v}}_h), \bar{0})}_v^2 = \mu
  \norm[0]{\varepsilon(\tilde{v}_h)}_{\mathcal{T}_h}^2 + \mu \eta
  \norm[0]{h_K^{-1}(\tilde{v}_h - m(\tilde{\bar{v}}_h))}_{\partial
    \mathcal{T}_h}^2$. Using the triangle inequality and steps similar
  to those used to prove
  \cite[eq. (47)]{rhebergen2018preconditioning}, Korn's inequality
  \cite{brenner2004korn}, using that $\eta > 1$, and
  \cref{eq:pT0vtilde_b},
  \begin{equation}
    \label{eq:vtildemkvtilde}
    \mu^{1/2} \norm[0]{h_K^{-1/2}(\tilde{v}_h-m_K(\tilde{\bar{v}}_h))}_{\partial \mathcal{T}_h}
    \le c \tnorm{\tilde{\boldsymbol{v}}_h}_v
    \le c \tnorm{(p_{T}^0,\bar{0})}_{q_T},
  \end{equation}
  so that
  $\tnorm{(\tilde{v}_h - m(\tilde{\bar{v}}_h), \bar{0})}_v \le c
  \tnorm{(p_{T}^0,\bar{0})}_{q_T}$. Then, using this in
  \cref{eq:dhu0bar0bound},
  $|d_h((u^0, \bar{0}), (\tilde{v}_h - m(\tilde{\bar{v}}_h),
  \bar{0}))| \le c \tnorm{(u^0,\bar{0})}_v\tnorm{(p_{T}^0,
    \bar{0})}_{q_T}$. Combining this result with \cref{eq:dhu0bar0} we
  obtain \cref{eq:proofcomplete-b}.
  
  Finally, we note that
  \begin{equation}
    \label{eq:p0p0ineq}
    \begin{split}
      ((c_0 + \alpha^2 \lambda^{-1})p^0, p^0)_{\mathcal{T}_h}
      &\le c_0 \norm[0]{p^0}_{\mathcal{T}_h}^2
        + 2 \lambda^{-1} \norm[0]{\alpha p^0 - p_T^0}_{\mathcal{T}_h}^2
        + 2 \lambda^{-1}\norm[0]{p_T^0}_{\mathcal{T}_h}^2
      \\
      &\le c_0 \norm[0]{p^0}_{\mathcal{T}_h}^2
        + 2 \lambda^{-1} \norm[0]{\alpha p^0 - p_T^0}_{\mathcal{T}_h}^2
        + 2 c_l\tnorm{(p_T^0,\bar{0})}_{q_T}^2,
    \end{split}
  \end{equation}
  where we used that $\mu \lambda^{-1}\le c_l$. Combining
  \cref{eq:result1,eq:proofcomplete-ab,eq:p0p0ineq} and using the
  definitions of the norms completes the proof of
  \cref{eq:desired-estimate-0}.
  \\
  \textbf{Step 3.} To prove \cref{eq:A3formu1}, consider
  \cref{eq:local-ah-reform-eqs} for $(u^1, p_T^1, z^1, p^1)$. Choosing
  $(v',q_T',w',q')= (u^1-m_K(\bar{v}_h),-p_T^1,z^1,-p^1)$ in the
  equations, and rearranging, we find that
  \begin{align*}
    & \mu \norm[0]{\varepsilon(u^1)}_K^2
      + 2\mu \langle \varepsilon(u^1)n, \bar{v}_h-u^1 \rangle_{\partial K}
      + \eta\mu h_K^{-1} \norm[0]{u^1-\bar{v}_h}_{\partial K}^2
    \\
    & + \lambda^{-1} \norm[0]{\alpha p^1 - p_T^1}_K^2 + (\kappa^{-1} z^1, z^1)_K + (c_0 p^1, p^1)_K 
    \\
    &=
      \mu \langle \varepsilon(u^1)n, \bar{v}_h-m_K(\bar{v}_h) \rangle_{\partial K}
      + \eta\mu h_K^{-1} \langle \bar{v}_h-m_K(\bar{v}_h), \bar{v}_h-u^1 \rangle_{\partial K}
    \\
    &\quad - \langle \bar{q}_{Th}, (u^1-\bar{v}_h) \cdot n\rangle_{\partial K}
      - \langle \bar{q}_{Th}, (\bar{v}_h-m_K(\bar{v}_h)) \cdot n\rangle_{\partial K}
      - \langle \bar{q}_h - m_K(\bar{q}_h), z^1\cdot n \rangle_{\partial K}
    \\
    &\le \tfrac{c_d}{2} (\mu \norm[0]{\varepsilon(u^1)}_K^2
      + \eta\mu h_K^{-1} \norm[0]{u^1-\bar{v}_h}_{\partial K}^2)
      + c \eta \mu h_K^{-1} \norm[0]{\bar{v}_h - m_K(\bar{v}_h)}_{\partial K}^2
      + c \mu^{-1} \eta^{-1} \norm[0]{h_K^{1/2}\bar{q}_{Th}}_{\partial K}^2
    \\
    &\quad + c\kappa h_K^{-1} \norm[0]{\bar{q}_h - m_K(\bar{q}_h)}_{\partial K}^2
      + \tfrac{1}{2} (\kappa^{-1} z^1, z^1)_K,
  \end{align*}
  where for the inequality we used the Cauchy--Schwarz inequality, a
  discrete trace inequality, and Young's inequalities. The above
  estimate, with the summation over $K \in \mathcal{T}_h$ and
  \cref{eq:stabdh}, leads to
  \begin{equation}
    \label{eq:u1z1p1p1bound}
    \begin{split}
      &\tnorm{(u^1, \bar{v}_h)}_v^2
        + \tnorm{z^1}_{w}^2
        + \lambda^{-1} \norm[0]{\alpha p^1 - p_T^1}_{\mathcal{T}_h}^2
        + (c_0 p^1, p^1)_{\mathcal{T}_h} 
      \\
      &\le c (\eta\mu\tnorm{\bar{v}_h}_{h,u}^2
        + \mu^{-1} \eta^{-1} \norm[0]{h_K^{1/2}\bar{q}_{Th}}_{\partial \mathcal{T}_h}^2
        + \kappa \tnorm{\bar{q}_h}_{h,p}^2)
      \\
      &\le c \del[2]{
        \eta \mu \tnorm{\bar{v}_h}_{h,u}^2
        + \mu^{-1} \eta^{-1} \norm[0]{h_K^{1/2}\bar{q}_{Th}}_{\partial\mathcal{T}_h}^2
        + \tilde{a}_h((\tilde{l}_p(\bar{q}_h),\bar{q}_h),(\tilde{l}_p(\bar{q}_h),\bar{q}_h)) },      
    \end{split}
  \end{equation}
  where we used \cref{eq:tildeah-estimate} to obtain the last
  inequality. Furthermore, using similar steps used to prove
  \cref{eq:proofcomplete-a}, it holds that
  \begin{equation}
    \label{eq:p1qbarmqbarbound1p}
    \kappa^{1/2} \tnorm{(p^1, \bar{q}_h -m(\bar{q}_h))}_{1,p} \le 
    c \eta^{1/2} \tnorm{z^1}_w.
  \end{equation}
  We next prove
  \begin{equation}
    \label{eq:pT1barqqTboundu1v}
    \tnorm{(p_T^1,\bar{q}_{Th})}_{q_T}
    \le c\del[2]{\tnorm{(u^1, \bar{v}_h)}_v + \mu^{-1/2}\eta^{-1/2}\norm[0]{h_K^{1/2}\bar{q}_{Th}}_{\partial \mathcal{T}_h}}.
  \end{equation}
  Since
  $\tnorm{(p_T^1, \bar{q}_{Th})}_{q_T}^2 = \mu^{-1}
  \norm[0]{p_T^1}_{\mathcal{T}_h}^2 + \mu^{-1} \eta^{-1}
  \norm[0]{h_K^{1/2} \bar{q}_{Th}}_{\partial
    \mathcal{T}_h}^2$ it is sufficient to prove
  \begin{equation}
    \label{eq:pT1barqqTboundu1v-b}
    \mu^{-1/2} \norm[0]{p_T^1}_{\mathcal{T}_h}
    = \tnorm{(p_T^1, \bar{0})}_{q_T}
    \le c\del[2]{\tnorm{(u^1, \bar{v}_h)}_v + \mu^{-1/2} \eta^{-1/2} \norm[0]{h_K^{1/2}\bar{q}_{Th}}_{\partial \mathcal{T}_h}}.
  \end{equation}
  To show this, we note that a consequence of \cref{eq:stabbhQh} is
  that for given $(p_T^1, \bar{0}) \in \boldsymbol{Q}_h$ there exists
  $\tilde{\boldsymbol{v}}_h = (\tilde{v}_h, \tilde{\bar{v}}_h) \in
  \boldsymbol{V}_h$ such that
  \begin{subequations}
    \begin{align}
      \label{eq:pT0vtilde_ann}
      \tnorm{(p_T^1, \bar{0})}_{q_T}^2
      & = b_h(\tilde{v}_h, (p_T^1,\bar{0})),
      \\
      \label{eq:pT0vtilde_bnn}
      \tnorm{\tilde{\boldsymbol{v}}_h}_v
      &\le c \tnorm{(p_T^1, \bar{0})}_{q_T}.
    \end{align}
  \end{subequations}
  From \cref{eq:local-ah-reform-eq1} we note that $u^1,p_T^1$ satisfy:
  \begin{multline*}
    (\mu \varepsilon(u^1), \varepsilon({v}'))_{K}
    +\mu \eta h_K^{-1}\langle (u^1 - \bar{v}_h),{v}'\rangle_{\partial K}
    - \langle \mu \varepsilon(u^1)n, {v}'\rangle_{\partial K}
    - \langle \mu \varepsilon({v}')n, u^1-\bar{v}_h\rangle_{\partial K}
    \\
    - (p_T^1, \nabla \cdot v')_K
    = - \langle \bar{q}_{Th}, v' \cdot n\rangle_{\partial K}.
  \end{multline*}
  Summing over all $K$ results in
  $d_h((u^1,\bar{v}_h),(v',\bar{0})) + b_h(v',(p_T^1,\bar{q}_{Th})) =
  0$. Choosing $v'|_K = \tilde{v}_h - m_K(\tilde{\bar{v}}_h)$ and
  using \cref{eq:pT0vtilde_ann}, we find:
  \begin{equation}
    \label{eq:dhu1barvheqpT1norm}
    d_h((u^1,\bar{v}_h),(\tilde{v}_h - m(\tilde{\bar{v}}_h),\bar{0}))
    + b_h(\tilde{v}_h - m(\tilde{\bar{v}}_h),(0,\bar{q}_{Th}))
    =
    - b_h(\tilde{v}_h,(p_T^1,\bar{0}))           
    =
    - \tnorm{(p_T^1, \bar{0})}_{q_T}^2.
  \end{equation}
  By \cref{eq:dhbound} we have
  \begin{equation}
    \label{eq:dhu1bar1bound}
    \begin{split}
      |d_h((u^1, \bar{v}_h), (\tilde{v}_h - m(\tilde{\bar{v}}_h), \bar{0}))|
      &\le c \tnorm{(u^1, \bar{v}_h)}_v \tnorm{(\tilde{v}_h - m(\tilde{\bar{v}}_h), \bar{0})}_v
      \\
      &\le c \tnorm{(u^1, \bar{v}_h)}_v \tnorm{(p_T^1, \bar{0})}_{q_T},
    \end{split}
  \end{equation}
  where the second inequality follows the same steps as used in Step 2
  to prove
  $\tnorm{(\tilde{v}_h - m(\tilde{\bar{v}}_h), \bar{0})}_v \le c
  \tnorm{(p_T^1, \bar{0})}_{q_T}$ but now using
  \cref{eq:pT0vtilde_bnn} instead of \cref{eq:pT0vtilde_b}. By
  \cref{eq:bhboundedness1} we have
  \begin{equation}
    \label{eq:bhu1bar1bound}
    \begin{split}
      |b_h(\tilde{v}_h - m(\tilde{\bar{v}}_h), (0,\bar{q}_{Th}) )|
      &\le c \tnorm{(\tilde{v}_h - m(\tilde{\bar{v}}_h),\bar{0})}_v \tnorm{(0,\bar{q}_{Th})}_{q_T}
      \\
      &\le c \tnorm{(p_T^1, \bar{0})}_{q_T} \mu^{-1/2} \eta^{-1/2} \norm[0]{h_K^{1/2}\bar{q}_{Th}}_{\partial
        \mathcal{T}_h},
    \end{split}
  \end{equation}
  where for the second inequality we again used that
  $\tnorm{(\tilde{v}_h - m(\tilde{\bar{v}}_h), \bar{0})}_v \le c
  \tnorm{(p_T^1, \bar{0})}_{q_T}$ and where we used the definition of
  $\tnorm{(0,\bar{q}_{Th})}_{q_T}$. \Cref{eq:pT1barqqTboundu1v-b} now
  follows by combining
  \cref{eq:dhu1barvheqpT1norm,eq:dhu1bar1bound,eq:bhu1bar1bound}.

  Finally, as in \cref{eq:p0p0ineq} and using
  $\mu \lambda^{-1}\le c_l$, we note that
  \begin{equation}
    \label{eq:p1p1boundp1p1}
    ((c_0 + \alpha^2 \lambda^{-1})p^1, p^1)_{\mathcal{T}_h}
    \le c_0 \norm[0]{p^1}_{\mathcal{T}_h}^2
    + 2\lambda^{-1} \norm[0]{\alpha p^1 - p_T^1}_{\mathcal{T}_h}^2
    + 2c_l \tnorm{(p_T^1,\bar{0})}_{q_T}^2.
  \end{equation}
  Combining
  \cref{eq:u1z1p1p1bound,eq:p1qbarmqbarbound1p,eq:pT1barqqTboundu1v,eq:p1p1boundp1p1}
  and using the definitions of the norms completes the proof of
  \cref{eq:A3formu1}.
\end{proof}

\end{document}